\documentclass{amsart}
\usepackage{amssymb}
\usepackage{latexsym}
\usepackage{euscript}

\def\sfi{{\text{\rm{\textsf i}}}}

\def\sD{{\mathfrak D}}      
   \def\sH{{\mathfrak H}}   
      \def\sL{{\mathfrak L}}
\def\sM{{\mathfrak M}}   \def\sN{{\mathfrak N}}

      \def\dC{{\mathbb C}}

      \def\dR{{\mathbb R}}

   \def\dZ{{\mathbb Z}}

      \def\cC{{\EuScript C}}

\def\wt#1{{{\widetilde #1} }}

\def\wh#1{{{\widehat #1} }}

\def\bm\chi{\mbox{\boldmath$\chi$}}

\def\RE{{\rm Re\,}}

\def\ker{{\rm ker\,}}
\def\ran{{\rm ran\,}}
\def\cran{{\rm \overline{ran}\,}}
\def\dom{{\rm dom\,}}
\def\mul{{\rm mul\,}}

\def\cdom{{\rm \overline{dom}\,}}
\def\clos{{\rm clos\,}}
\def\col{{\rm col\,}}
\def\dim{{\rm dim\,}}

\let\xker=\ker \def\ker{{\xker\,}}
\def\spn{{\rm span\,}}

\def\sgn{{\rm sgn\,}}
\def\sign{{\rm sign\,}}

\def\uphar{{\upharpoonright\,}}

\newcommand\tA{\widetilde{A}}

\def\Ext{{\rm Ext\,}}

\newcommand {\sk}[3]{\left#1#2\right#3}  
\renewcommand {\l}{\lambda}
\renewcommand {\k}{\kappa}

\newcommand {\e}{\varepsilon}
\newcommand {\ov}{\overline}

\renewcommand {\r}{\rho}
\newcommand\ff{\varphi}
\newcommand {\p}{\perp}

\newtheorem{theorem}{Theorem}[section]
\newtheorem{proposition}[theorem]{Proposition}
\newtheorem{corollary}[theorem]{Corollary}
\newtheorem{lemma}[theorem]{Lemma}

\theoremstyle{definition}

\newtheorem{remark}[theorem]{Remark}
\newtheorem{definition}[theorem]{Definition}

\numberwithin{equation}{section}

\begin{document}

\title[Completion, extension, factorization, and lifting of operators]
{Completion, extension, factorization, and lifting of operators with
a negative index}
\author{D.~Baidiuk}
\author{S.~Hassi}

\address{Department of Mathematics and Statistics \\
University of Vaasa \\
P.O. Box 700, 65101 Vaasa \\
Finland} \email{dbaidiuk@uwasa.fi}

\address{Department of Mathematics and Statistics \\
University of Vaasa \\
P.O. Box 700, 65101 Vaasa \\
Finland} \email{sha@uwasa.fi}





\begin{abstract}
The famous results of M.G. Kre\u{\i}n concerning the description of
selfadjoint contractive extensions of a Hermitian contraction $T_1$
and the characterization of all nonnegative selfadjoint extensions
$\wt A$ of a nonnegative operator $A$ via the inequalities $A_K\le
\wt A \le A_F$, where $A_K$ and $A_F$ are the Kre\u{\i}n-von Neumann
extension and the Friedrichs extension of $A$, are generalized to
the situation, where $\wt A$ is allowed to have a fixed number of
negative eigenvalues. These generalizations are shown to be possible
under a certain minimality condition on the negative index of the
operators $I-T_1^*T_1$ and $A$, respectively; these conditions are
automatically satisfied if $T_1$ is contractive or $A$ is
nonnegative, respectively.

The approach developed in this paper starts by establishing first a
generalization of an old result due to Yu.L. Shmul'yan on
completions of $2\times 2$ nonnegative block operators. The
extension of this fundamental result allows us to prove analogs of
the above mentioned results of M.G. Kre\u{\i}n and, in addition, to
solve some related lifting problems for $J$-contractive operators in
Hilbert, Pontryagin and Kre\u{\i}n spaces in a simple manner. Also
some new factorization results are derived, for instance, a
generalization of the well-known Douglas factorization of Hilbert
space operators. In the final steps of the treatment some very
recent results concerning inequalities between semibounded
selfadjoint relations and their inverses turn out to be central in
order to treat the ordering of non-contractive selfadjoint operators
under Cayley transforms properly.
\end{abstract}

\maketitle


\section{Introduction}

Almost 70 years ago in his famous paper \cite{Kr47} M.G.~Kre\u{\i}n
proved that for a densely defined nonnegative operator $A$ in a
Hilbert space there are two extremal extensions of $A$, the
Friedrichs (hard) extension $A_F$ and the Kre\u{\i}n-von Neumann
(soft) extension $A_K$, such that every nonnegative selfadjoint
extension $\wt A$ of $A$ can be characterized by the following two
inequalities:
\[
 (A_F+a)^{-1}\le (\tA+a)^{-1}\le (A_K+a)^{-1},
 \quad a>0.
\]
To obtain such a description he used Cayley transforms of the form
\[
 T_1=(I-A)(I+A)^{-1} \quad T=(I-\wt A)(I+\wt A)^{-1},
\]
to reduce the study of unbounded operators to the study of
contractive selfadjoint extensions $T$ of a Hermitian nondensely
defined contraction $T_1$. In the study of contractive selfadjoint
extensions of $T_1$ he introduced a notion which is nowadays called
``the shortening of a bounded nonnegative operator $H$ to a closed
subspace $\sN$'' of $\sH$ as the (unique) maximal element in the set
\[
  \{\,D \in [\sH] :\, 0 \leq D \leq H, \, \ran D \subset \sN\,\},
\]
which is denoted by $H_\sN$; cf. \cite{AD,AT,Pek}. Using this notion
he proved the existence of a minimal and maximal contractive
extension $T_m$ and $T_M$ of $T_1$ and that $T$ is a selfadjoint
contractive extension of $T_1$ if and only if $T_m\le T\le T_M$,
more explicitly that $T=T_m+(I+T)\sN$ and $T=T_M-(I-T)\sN$ when
$\sN=\sH\ominus\dom T_1$.

Later the study of nonnegative selfadjoint extensions of $A\ge 0$
was generalized to the case of nondensely defined operators $A\ge 0$
by T.~Ando and K.~Nishio \cite{AN}, as well as to the case of linear
relations (multivalued linear operators) $A\ge 0$ by E.A.~Coddington
and H.S.V.~de~Snoo \cite{CS}. Further studies followed this work of
M.G. Kre\u{\i}n; the approach in terms of ``boundary conditions'' to
the extensions of a positive operator $A$ was proposed by
M.I.~Vishik \cite{V} and M.S.~Birman \cite{B}; an exposition of this
theory based on the investigation of quadratic forms can be found
from \cite{AS}. An approach to the extension theory of symmetric
operators based on abstract boundary conditions was initiated even
earlier by J.W. Calkin \cite{Cal39} under the name of reduction
operators, and later, independently the technique of boundary
triplets was introduced to formalize the study of boundary value
problems in the framework of general operator theory; see
\cite{Koch,Bruk,GG,DM1,MMM2,DM2}. Later the extension theory of
unbounded symmetric Hilbert space operators and related resolvent
formulas originating also from the work of M.G. Kre\u{\i}n
\cite{Kr44,Kr46}, see also e.g. \cite{LT}, was generalized to the
spaces with indefinite inner products in the well-known series of
papers by H. Langer and M.G. Kre\u{\i}n, see e.g. \cite{KL1,KL2},
and all of this has been further investigated, developed, and
extensively applied in various other areas of mathematics and
physics by numerous other researchers.

In spite of the long time span, natural extensions of the original
result of M.G. Kre\u{\i}n in \cite{Kr47} have not occurred in the
literature. Obviously the most closely related result appears in
\cite{CG92}, where for a given pair of a row operator
$T_r=(T_{11},T_{12})\in[\sH_1\oplus\sH_1',\sH_2]$ and a column
operator $T_c=\col(T_{11},T_{21})\in[\sH_1,\sH_2\oplus\sH_2']$ the
problem for determining all possible operators $\wt
T\in[\sH_1\oplus\sH_1',\sH_2\oplus\sH_2']$ acting from the Hilbert
space $\sH_1\oplus\sH_1'$ to the Hilbert space $\sH_2\oplus\sH_2'$
such that
\[
 P_{\sH_2}\wt T=T_r, \quad \wt T\uphar \sH_1=T_c,
\]
and such that the following negative index (number of negative eigenvalues)
conditions are satisfied
\[
 \kappa_1:=\nu_-(I-\wt T^*\wt T)=\nu_-(I-T_c^*T_c),\quad
 \kappa_2:=\nu_-(I-\wt T\wt T^*)=\nu_-(I-T_rT_r^*),
\]
is considered. The problem was solved in \cite[Theorem~5.1]{CG92}
under the condition $\kappa_1,\kappa_2<\infty$. In the literature
cited therein appears also a reference to an unpublished manuscript
by H. Langer and B. Textorius with the title ``Extensions of a
bounded Hermitian operator $T$ preserving the number of negative
squares of $I-T^*T$'', where obviously a similar problem for a given
bounded Hermitian (column operator) $T$ has been investigated; see
\cite[Section~6]{CG92}. However, in \cite{CG92} the existence of
possible extremal extensions in the solution set in the spirit of
\cite{Kr47}, when it is nonempty, have not been investigated. Also
possible investigations of analogous results for unbounded symmetric
operators with a fixed negative index seem to be unavailable in the
literature.

In this paper we study specific classes of such
``quasi-contractive'' bounded symmetric operators $T_1$ with
$\nu_-(I-T_1^*T_1)<\infty$ as well as ``quasi-nonnegative''
operators $A$ with $\nu_-(A)<\infty$ and the existence and
description of all possible selfadjoint extensions $T$ and $\wt A$
of them which preserve the given negative indices
$\nu_-(I-T^2)=\nu_-(I-T_1^*T_1)$ and $\nu_-(\wt A)=\nu_-(A)$,
respectively, under a further minimality condition on the negative
index $\nu_-(I-T_1^*T_1)$ and $\nu_-(A)$. Under such conditions it
is shown that if there is a solution then there are again two
extremal extensions which then describe the whole solution set via
two operator inequalities, just as in the original paper of
M.G.~Kre\u{\i}n. The approach developed in this paper differs from
the approach in \cite{Kr47}. In fact, the approach used in a recent
paper of Hassi, Malamud and de Snoo \cite{HMS04}, a technique
appearing also in an earlier paper of Kolmanovich and Malamud
\cite{KM1}, will be successfully generalized. In \cite{HMS04} the
original results of M.G.~Kre\u{\i}n have been proved in the general
setting of a not necessarily densely defined nonnegative operator
and, more generally, for a nonnegative linear relation $A$.

The starting point in our approach is to establish a generalization
of an old result due to Yu.L. Shmul'yan \cite{S59} on completions of
$2\times 2$ nonnegative block operators where the result was applied
for introducing so-called Hellinger operator integrals. Our
extension of this fundamental result is given in Section \ref{sec1};
see Theorem \ref{T:1} (for the case $\kappa<\infty$) and Theorem
\ref{T:1ext} (for the case $\kappa=\infty$). Obviously, these
results can be considered to be the most important inventions in the
present paper and it is possible that several further applications
for them will occur in forthcoming literature.

In this paper we will extensively apply Theorem \ref{T:1}. In
Section \ref{sec2} this result is specialized to a class of $2\times
2$ block operators to characterize occurrence of a minimal negative
index for the so-called Schur complement of the block operator, see
Theorem \ref{thmB}. This result can be viewed also as a
factorization result and, in fact, it yields a generalization of the
well-known Douglas factorization of Hilbert space operators in
\cite{Douglas}, see Proposition \ref{BKcor1}, which is completed by
a generalization of Sylvester's criterion on additivity of inertia
on Schur complements in Proposition~\ref{sylvester}. In
Section~\ref{sec3} Theorem \ref{T:1}, or its special case Theorem
\ref{thmB}, is applied to solve lifting problems for $J$-contractive
operators in Hilbert, Pontryagin and Kre\u{\i}n spaces in a new
simple way, the most general version of which is formulated in
Theorem~\ref{Lifthm}: this result was originally proved in
\cite[Theorem~2.3]{CG89} with the aid of \cite[Theorem~5.3]{ACG87};
for a special case, see also \cite{Drit90,DritRov90}. In the Hilbert
space case this problem has been solved in \cite{AG82,DaKaWe,ShYa},
further proofs and facts can be found e.g. from
\cite{Ar2006,AHS2007,BN2,KM1,MMM3}.

Section \ref{sec4} contains the extension of the fundamental result
of M.G.~Kre\u{\i}n in \cite{Kr47}, see Theorem~\ref{T:contr}, which
characterizes the existence and gives a description of all
selfadjoint extension $T$ of a bounded symmetric operator $T_1$
satisfying the following minimal index condition
$\nu_-(I-T^2)=\nu_-(I-T_{11}^2)$ by means of two extreme extensions
via $T_m\le T\le T_M$. In Section~\ref{sec5} selfadjoint extensions
of unbounded symmetric operators, and symmetric relations, are
studied under a similar minimality condition on the negative index
$\nu_-(A)$; the main result there is Theorem~\ref{KreinThm}. It is a
natural extension of the corresponding result of M.G.~Kre\u{\i}n in
\cite{Kr47}. The treatment here uses Cayley transforms and hence is
analogous to that in \cite{Kr47}. However, the existence of two
extremal extensions in this setting and the validity of all the
operator inequalities appearing therein depend essentially of very
recent ``antitonicity results'' proved for semibounded selfadjoint
relations in \cite{BHSW2014} concerning correctness of the
implication $H_1\le H_2$ $\Rightarrow$ $H_2^{-1}\le H_1^{-1}$ in the
case that $H_1$ and $H_2$ have some finite negative spectra. In this
section also an analog of the so-called Kre\u{\i}n's uniqueness
criterion for the equality $T_{m}=T_{M}$ is established.


\section{A completion problem for block operators}
\label{sec1}

By definition the modulus $|C|$ of a closed operator $C$ is the
nonnegative selfadjoint operator $|C|=(C^*C)^{1/2}$. Every closed
operator admit a polar decomposition $C=U|C|$, where $U$ is a
(unique) partial isometry with the initial space $\cran |C|$ and the
final space $\cran C$, cf. \cite{Kato}. For a selfadjoint operator
$H=\int_{\dR} t\, dE_t$ in a Hilbert space $\sH$ the partial
isometry $U$ can be identified with the signature operator, which
can be taken to be unitary: $J=\sign(H)=\int_{\dR}\,\sign(t)\,dE_t$,
in which case one should define $\sign(t)=1$ if $t\ge 0$ and
otherwise $\sign(t)=-1$.

\subsection{Completion to operator blocks with finite negative
index.}

The following theorem solves a completion problem for a bounded
incomplete block operator $A^0$ of the form
\begin{equation}\label{A0}
 A^0=
  \begin{pmatrix}
   A_{11}&A_{12}\\
   A_{21}&\ast
  \end{pmatrix}
  \begin{pmatrix}
   \sH_1\\
   \sH_2
  \end{pmatrix}
  \to
  \begin{pmatrix}
   \sH_1\\
   \sH_2
  \end{pmatrix}
\end{equation}
in the Hilbert space $\sH=\sH_1\oplus\sH_2$.

\begin{theorem}\label{T:1}
Let $\sH=\sH_1\oplus\sH_2$ be an orthogonal decomposition of the
Hilbert space $\sH$ and let $A^0$ be an incomplete block operator of
the form \eqref{A0}. Assume that $A_{11}=A_{11}^*$ and
$A_{21}=A_{12}^*$ are bounded, $\nu_-(A_{11})=\k<\infty$, where
$\k\in\mathbb{Z}_+$, and let $J=\sign(A_{11})$ be the (unitary)
signature operator of $A_{11}$. Then:
\begin{enumerate}\def\labelenumi{\rm(\roman{enumi})}
\item
 There exists a completion $A\in[\sH]$ of $A^0$ with some
 operator $A_{22}=A_{22}^*\in[\sH_2]$ such that
 $\nu_-(A)=\nu_-(A_{11})=\k$ if and only if
 \begin{equation}\label{E:1}
  \ran A_{12}\subset\ran |A_{11}|^{1/2}.
 \end{equation}

\item
In this case the operator $S=|A_{11}|^{[-1/2]}A_{12}$, where
$|A_{11}|^{[-1/2]}$ denotes the (generalized) Moore-Penrose
inverse of $|A_{11}|^{1/2}$, is well defined and
$S\in[\sH_2,\sH_1]$. Moreover, $S^*JS$ is the smallest operator
in the solution set
\begin{equation}\label{E:sol}
 \mathcal{A}:=\sk\{{A_{22}=A_{22}^*\in[\sH_2]:\, A=(A_{ij})_{i,j=1}^{2}:\nu_-(A)=\k
 }\}
\end{equation}
and this solution set admits a description as the (semibounded)
operator interval given by
\[
  \mathcal{A}=\sk\{{A_{22}\in[\sH_2]:\, A_{22}=S^*JS+Y,\, Y=Y^*\ge 0}\}.
\]
\end{enumerate}
\end{theorem}
\begin{proof}
(i) Assume that there exists a completion $A_{22}\in\mathcal{A}$.
Let $\l_\k\leq\l_{\k-1}\leq...\leq\l_1<0$ be all the negative
eigenvalues of $A_{11}$ and let $\e$ be such that $|\l_1|>\e>0$.
Then $0\in\r(A_{11}+\e)$ and hence one can write
 \begin{equation}\label{E:Silv}
 \begin{split}
  &\begin{pmatrix}
   I&0\\
   -A_{21}(A_{11}+\e)^{-1}&I
  \end{pmatrix}
  \begin{pmatrix}
   A_{11}+\e&A_{12}\\
   A_{21}&A_{22}+\e
  \end{pmatrix}
  \begin{pmatrix}
   I&-(A_{11}+\e)^{-1}A_{12}\\
   0&I
  \end{pmatrix}\\=
  &\begin{pmatrix}
   A_{11}+\e&0\\
   0&A_{22}+\e-A_{21}(A_{11}+\e)^{-1}A_{12}
  \end{pmatrix}
 \end{split}
 \end{equation}
The operator in the righthand side of \eqref{E:Silv} has $\k$
negative eigenvalues if and only if
\begin{equation}\label{E:2}
  A_{21}(A_{11}+\e)^{-1}A_{12}\leq A_{22}+\e
\end{equation}
or equivalently
 \begin{equation}\label{E:3}
  \int\limits_{-\|A_{11}\|}^{\|A_{11}\|}(t+\e)^{-1}d\|E_tA_{12}f\|^2\leq\e\|f\|^2+(A_{22}f,f),
 \end{equation}
where $E_t$ is the spectral family of $A_{11}$. We rewrite
\eqref{E:3} in the form
\[
  \int_{[-\|A_{11}\|,0)} (t+\e)^{-1}d\|E_tA_{12}f\|^2+
  \int_{[0,\|A_{11}\|]}(t+\e)^{-1}d\|E_tA_{12}f\|^2\leq\e\|f\|^2+(A_{22}f,f),
\]
This yields the estimate
\begin{equation}\label{E:4}
  \int_{[0,\|A_{11}\|]} (t+\e)^{-1}d\|E_tA_{12}f\|^2\leq\e\|f\|^2+(A_{22}f,f)-\frac{1}{\l_1+\e}\|A_{12}f\|^2.
\end{equation}

By letting $\e\searrow 0$ in \eqref{E:4} the monotone convergence
theorem implies that
\[
 P_+A_{12}f\in\ran A_{11+}^{1/2}\subset \ran |A_{11}|^{1/2}
\]
for all $f\in\sH_2$; here $A_{11+}=\int_{[0,\|A_{11}\|]} t\, dE_t$
stands for the nonnegative part of $A_{11}$ and $P_+$ is the
orthogonal projection onto the corresponding closed subspace $\cran
A_{11+}=\int_{[0,\|A_{11}\|]} \, dE_t$. This implies that $\ran
A_{12}\subset\ran |A_{11}|^{1/2}$.

Conversely, if $\ran A_{12}\subset\ran |A_{11}|^{1/2}$, then the
operator $S:=|A_{11}|^{[-1/2]}A_{12}$ is well defined, closed and
bounded, i.e., $S\in [\sH_2,\sH_1]$. Since $A_{12}=|A_{11}|^{1/2}S$,
it follows from $A_{21}=S^*|A_{11}|^{1/2}$ and
\begin{equation}\label{Amin}
 A=
 \begin{pmatrix}
  |A_{11}|^{1/2}\\
  S^*J
 \end{pmatrix}J
 \begin{pmatrix}
  |A_{11}|^{1/2}&JS\\
 \end{pmatrix}:\ \nu_-(A)=\k,
\end{equation}
that the operator $A_{22}=S^*JS$ gives a completion for $A^0$.

(ii) According to (i) $A_{21}=S^*|A_{11}|^{1/2}$, and
$S^*JS\in[\sH_2]$ gives a solution to the completion problem
\eqref{A0}. Now
$$
s-\lim_{\e\searrow 0}A_{21}(A_{11}+\e)^{-1}A_{12}=s-\lim_{\e\searrow
0}S^*|A_{11}|^{1/2}(A_{11}+\e)^{-1}|A_{11}|^{1/2}S=S^*JS
$$
and if $A_{22}$ is an arbitrary operator in the set \eqref{E:sol},
then by letting $\e\searrow 0$ one concludes that $S^*JS\leq
A_{22}$. Therefore, $S^*JS$ satisfies the desired minimality
property.

To prove the last statement assume that $Y\in[\sH_2]$ and that $Y\ge
0$. Then $A_{22}=S^*JS+Y$ inserted in $A^0$ defines a block operator
$A_Y\ge A_{\rm{min}}$. In particular, $\nu_-(A_Y)\le
\nu_-(A_{\rm{min}})=\kappa<\infty$. On the other hand, it is clear
from the formula
\begin{equation}\label{AY}
  A_Y=
 \begin{pmatrix}
  |A_{11}|^{1/2}\\
  S^*J
 \end{pmatrix}J
 \begin{pmatrix}
  |A_{11}|^{1/2}&JS\\
 \end{pmatrix}
 +
 \begin{pmatrix}
  0 & 0 \\ 0 & Y \\
 \end{pmatrix}
\end{equation}
that the $\kappa$-dimensional eigenspace corresponding to the
negative eigenvalues of $A_{11}$ is $A_Y$-negative and, hence,
$\nu_-(A_Y)\ge \kappa$. Therefore, $\nu_-(A_Y)=\kappa$ and
$Y\in\mathcal{A}$.
\end{proof}

Notice that in the factorization $A_{12}=|A_{11}|^{1/2}S$,
$S$ is uniquely determined under the condition
$\ran S\subset \cran A_{11}$ (which implies that $\ker A_{12}=\ker S$);
cf. \cite{Douglas}.

In the case that $\kappa=0$, the result in Theorem~\ref{T:1} reduces
to the well-known criterion concerning completion of an incomplete
block operator to a nonnegative operator; cf. \cite{S59}. In the
case of matrices acting on a finite dimensional Hilbert space, the
result with $\kappa>0$ has been proved very recently in the appendix
of \cite{DHS2012}, where it was applied in solving indefinite
truncated moment problems. In the present paper Theorem~\ref{T:1}
will be one of the main tools for further investigations.

\subsection{Completion to operator blocks with an infinite negative
index.}

The completion result in Theorem~\ref{T:1} is of some general
interest already by the substantial number of its applications known
in the case of nonnegative operators. In this section the completion
problem is treated in the case that $\kappa=\infty$. For this
purpose some further notions will be introduced.

Recall that a subspace $\sM\subset \sH$ is said to be uniformly
$A$-negative, if there exists a positive constant $\nu>0$ such that
$(Af,f)\le -\nu \|f\|^2$ for all $f\in\sM$. It is maximal uniformly
$A$-negative, if $\sM$ has no proper uniformly $A$-negative
extension. The completion problem is now extended by claiming from
the completions the following maximality property:
\begin{equation}\label{Max}
 \text{there exists a subspace $\sM\subset \sH_1$ which is maximal uniformly $A$-negative}.
\end{equation}

\begin{theorem}\label{T:1ext}
Let $A^0$ be an incomplete block operator of the form \eqref{A0} in
the Hilbert space $\sH=\sH_1\oplus\sH_2$. Let $A_{11}=A_{11}^*$ and
$A_{21}=A_{12}^*$ be bounded, let $J=\sign(A_{11})$ be the (unitary)
signature operator of $A_{11}$, and, in addition, assume that there
is a spectral gap $(-\delta,0)\subset \rho(A_{11})$, $\delta>0$.
Then:
\begin{enumerate}\def\labelenumi{\rm(\roman{enumi})}
\item
 There exists a completion $A\in[\sH]$ of $A^0$ with some
 operator $A_{22}=A_{22}^*$ satisfying the condition \eqref{Max}
 if and only if
\[
  \ran A_{12}\subset\ran |A_{11}|^{1/2}
\]

\item
 In this case the operator $S=|A_{11}|^{[-1/2]}A_{12}$, where
 $|A_{11}|^{[-1/2]}$ denotes the (generalized) Moore-Penrose
 inverse of $|A_{11}|^{1/2}$, is well defined and
 $S\in[\sH_2,\sH_1]$. Moreover, $S^*JS$ is the smallest operator
 in the solution set
\[
  \mathcal{A}:=\sk\{{A_{22}=A_{22}^*\in[\sH_2]:\, A=(A_{ij})_{i,j=1}^{2} \textrm{ satisfies } \eqref{Max}
  }\}
\]
 and this solution set admits a description as the (semibounded)
  operator interval given by
\[
  \mathcal{A}=\sk\{{A_{22}\in[\sH_2]:\, A_{22}=S^*JS+Y,\, Y=Y^*\ge 0}\}.
\]
\end{enumerate}
\end{theorem}

\begin{proof}
To prove this result suitable modifications in the proof of
Theorem~\ref{T:1} are needed.

(i) First assume that $A_{22}\in\mathcal{A}$ gives a desired
completion for $A^0$. If $\e\in(0,\delta)$ then $0\in\r(A_{11}+\e)$
and therefore the block operator $(A_{ij})$ satisfies the formula
\eqref{E:Silv}. We claim that the condition \eqref{Max} implies the
inequality \eqref{E:2} for all sufficiently small values $\e>0$. To
see this let $\sM\subset \sH_1$ be a subspace for which the
condition \eqref{Max} is satisfied. Then $(A_{11}f,f)\le -\nu
\|f\|^2$ for some fixed $\nu>0$ and for all $f\in\sM$. Assume that
for some $0<\e_0<\max\{\nu,\delta\}$ \eqref{E:2} is not satisfied.
Then $((A_{22}+\e_0-A_{21}(A_{11}+\e_0)^{-1}A_{12})v_0,v_0)<0$ holds
for some vector $v_0\in\sH_2$.  Define
$\sL=W_{\e_0}^{-1}(\sM+\spn\{v_0\})$, where
\[
W_{\e_0}=  \begin{pmatrix}
   I&-(A_{11}+\e_0)^{-1}A_{12}\\
   0&I
  \end{pmatrix}.
\]
Clearly, $W_{\e_0}$ is bounded with bounded inverse and it maps
$\sM$ bijectively onto $\sM$, so that $\sL$ is a $1$-dimensional
extension of $\sM$. It follows from \eqref{E:Silv} that for all
$f\in\sL$,
\[
 (Af,f)+\e_0 \|f\|^2 = \left(\begin{pmatrix}
   A_{11}+\e_0 &0\\
   0&A_{22}+\e_0 -A_{21}(A_{11}+\e_0)^{-1}A_{12}
  \end{pmatrix} u,u \right)<0,
\]
where $u=W_{\e_0}f \in \sM+\spn\{v_0\}$. Therefore, $\sL$ is a
proper uniformly $A$-negative extension of $\sM$; a contradiction,
which shows that \eqref{E:2} holds for all
$0<\e<\max\{\nu,\delta\}$. Then, as in the proof of
Theorem~\ref{T:1} it is seen that $\ran A_{12}\subset \ran
|A_{11}|^{1/2}$; note that in the estimate \eqref{E:4} $\lambda_1$
is to be replaced by $-\delta$.

Conversely, if $\ran A_{12}\subset\ran |A_{11}|^{1/2}$, then
$S=|A_{11}|^{[-1/2]}A_{12}\in [\sH_2,\sH_1]$ and the block operator
$A$ in \eqref{Amin} gives a completion. To prove that $A$ satisfies
\eqref{Max} observe that if $\sM$ is a uniformly $A$-negative
subspace in $\sH$, then $\begin{pmatrix}
  |A_{11}|^{1/2}&JS\\
\end{pmatrix}$
maps it bijectively onto a uniformly $J$-negative
subspace in $\sH_1$. The spectral subspace corresponding to the
negative spectrum of $A_{11}$ is maximal uniformly $J$-negative in
$\sH_1$ and also uniformly $A$-negative in $\sH$. By the above
mapping property this subspace must be maximal uniformly
$A$-negative in $\sH$.

(ii) If $A_{22}=A_{22}^*$ defines a completion  $A\in[\sH]$ of $A^0$
such that \eqref{Max} is satisfied then by the proof of (i) the
inequality \eqref{E:2} holds for all sufficiently small values
$\e>0$. Now the minimality property of $S^*JS$ can be obtained in
the same manner as in Theorem~\ref{T:1}.

As to the last statement again for every $Y\in[\sH_2]$, $Y\ge 0$,
the block operator $A_Y$ defined in the proof of Theorem~\ref{T:1}
satisfies $A_Y\ge A_{\rm{min}}$. Hence, every uniformly
$A_Y$-negative subspace is also uniformly $A_{\rm{min}}$-negative.
Now it follows from the formula \eqref{AY} that the spectral
subspace corresponding to the negative spectrum of $A_{11}$, which
is maximal uniformly $A_{\rm{min}}$-negative, is also maximal
uniformly $A_Y$-negative. Hence, $A_Y$ satisfies \eqref{Max} and
$Y\in\mathcal{A}$.
\end{proof}

\section{Some factorizations of operators with finite negative index}
\label{sec2}

Theorems~\ref{T:1} and~\ref{T:1ext} contain a valuable tool
in solving a couple of other problems, which initially do not occur
as a completion problem of some symmetric incomplete block operator.
In this section it is shown that Theorem~\ref{T:1}
(a) can be used to characterize existence of certain $J$-contractive factorization
of operators via a minimal index condition;
(b) implies an extension of the well-known Douglas factorization result
with a certain specification to the Bogn\'ar-Kr\'amli factorization;
(c) yields an extension of a factorization result of Shmul'yan for $J$-bicontractions;
(d) allows an extension of a classical Sylvester law of inertia of a block operator,
which is originally used in characterizing nonnegativity of a bounded block operator
via Schur complement.

Some simple inertia formulas are now recalled.
The factorization $H=B^*EB$ clearly implies that $\nu_\pm(H)\le \nu_\pm(E)$.
If $H_1$ and $H_2$ are selfadjoint operators, then
\[
 H_1+H_2=\begin{pmatrix} I \\ I \end{pmatrix}^* \begin{pmatrix} H_1 & 0 \\ 0 & H_2 \end{pmatrix}
 \begin{pmatrix} I \\ I \end{pmatrix}
\]
shows that $\nu_\pm(H_1+H_2)\le \nu_\pm(H_1)+\nu_\pm(H_2)$. Consider
the selfadjoint block operator $H\in[\sH_1\oplus\sH_2]$ of the form
\begin{equation}\label{H}
 H=\begin{pmatrix} A & B^* \\ B & J_2 \end{pmatrix},
\end{equation}
where $J_2=J_2^*=J_2^{-1}$. By applying the above mentioned
inequalities shows that
\begin{equation}\label{minneg}
 \nu_\pm(A)\le \nu_\pm(A-B^*J_2B)+\nu_\pm(J_2).
\end{equation}
Assuming that $\nu_-(A-B^*J_2B)$ and $\nu_-(J_2)$ are finite,
the question when $\nu_-(A)$ attains its maximum in \eqref{minneg}, or equivalently,
$\nu_-(A-B^*J_2B)\ge \nu_-(A)-\nu_-(J_2)$ attains its minimum,
turns out to be of particular interest.
The next result characterizes this situation as
an application of Theorem~\ref{T:1}. Recall that if $A=J_A |A|$ is
the polar decomposition of $A$, then one can interpret $\sH_A=(\cran
A,J_A)$ as a Kre\u{\i}n space generated on $\cran A$ by the
fundamental symmetry $J_A=\sgn(A)$.

\begin{theorem}\label{thmB}
Let $A\in[\sH_1]$ be selfadjoint, $B\in[\sH_1,\sH_2]$, $J_2=J_2^*=J_2^{-1}\in[\sH_2]$,
and assume that $\nu_-(A),\nu_-(J_2)<\infty$. If the equality
\begin{equation}\label{min}
 \nu_-(A) = \nu_-(A-B^*J_2B)+\nu_-(J_2)
\end{equation}
holds, then $\ran B^*\subset \ran |A|^{1/2}$ and $B^*=|A|^{1/2}K$ for a unique operator
$K\in[\sH_2,\sH_A]$ which is $J$-contractive: $J_2-K^*J_A K\ge 0$.

Conversely, if $B^*=|A|^{1/2}K$ for some $J$-contractive operator
$K\in[\sH_2,\cran A]$, then the equality \eqref{min} is satisfied.
\end{theorem}
\begin{proof}
Assume that \eqref{min} is satisfied. The factorization
\[
 H=\begin{pmatrix} A & B^* \\ B & J_2 \end{pmatrix}
 = \begin{pmatrix} I & B^*J_2\\ 0 & I \end{pmatrix}
  \begin{pmatrix} A-B^* J_2B& 0 \\ 0 & J_2 \end{pmatrix}
   \begin{pmatrix} I  & 0 \\ J_2 B & I \end{pmatrix}
\]
shows that $\nu_-(H)=\nu_-(A-B^* J_2B)+\nu_-(J_2)$, which combined
with the equality \eqref{min} gives $\nu_-(H)=\nu_-(A)$. Therefore,
by Theorem~\ref{T:1} one has $\ran B^*\subset \ran |A|^{1/2}$ and
this is equivalent to the existence of a unique operator $K\in
[\sH_2,\cdom A]$ such that $B^*=|A|^{1/2}K$; i.e.
$K=|A|^{[-1/2]}B^*$. Furthermore, $K^*J_{A}K\leq J_2$ by the
minimality property of $K^*J_{A}K$ in Theorem~\ref{T:1}, in other
words $K$ is a $J$-contraction.

Converse, if $B^*=|A|^{1/2}K$ for some $J$-contraction $K\in
[\sH_2,\cdom A]$, then clearly $\ran B^*\subset\ran |A|^{1/2}$. By
Theorem~\ref{T:1} the completion problem for $H^{0}$ has solutions
with the minimal solution $S^*J_{A}S$, where
$S=|A|^{[-1/2]}B^*=|A|^{[-1/2]}|A|^{1/2}K=K$. Furthermore, by
$J$-contractivity of $K$ one has $K^*J_{A}K\le J_2$, i.e. $J_2$ is
also a solution and thus $\nu_-(H)=\nu_-(A)$ or, equivalently, the
equality \eqref{min} is satisfied.
\end{proof}

While Theorem~\ref{thmB} is obtained as a direct consequence of
Theorem \ref{T:1} it will be shown in the next section that this
result yields simple solutions to a wide class of lifting problems
for contractions in Hilbert, Pontryagin and Kre\u{\i}n space
settings.

Before deriving the next result some inertia formulas for
a class of selfadjoint block operators are recalled.
Consider the following two representations
\begin{equation*}
\begin{split}
\begin{pmatrix}
 J_1& T^*\\
 T & J_2
\end{pmatrix}&=
\begin{pmatrix}
 I&0\\
 TJ_1&I
\end{pmatrix}
\begin{pmatrix}
 J_1&0\\
 0&J_2-TJ_1T^*
\end{pmatrix}
\begin{pmatrix}
 I&J_1T^*\\
 0&I
\end{pmatrix}\\
&=\begin{pmatrix}
 I&T^*J_2\\
 0&I
\end{pmatrix}
\begin{pmatrix}
 J_1-T^*J_2 T&0\\
 0& J_2
\end{pmatrix}
\begin{pmatrix}
 I&0\\
 J_2T&I
\end{pmatrix},
\end{split}
\end{equation*}
where $J_i=J_i^*=J_i^{-1}$, $i=1,2$. Since here the triangular
operators are bounded with bounded inverse, one concludes that $\ran
(J_2-TJ_1T^*)$ is closed if and only if $\ran(J_1-T^*J_2 T)$ is
closed. Furthermore, one gets the following inertia formulas; cf.
e.g. \cite[Proposition~3.1]{ACG87}.

\begin{lemma}\label{inertia} With the above notations one has
\[
 \nu_\pm(J_1-T^*J_2T)+\nu_\pm(J_2)=\nu_\pm(J_2-TJ_1T^*)+\nu_\pm(J_1),
\]
\[
 \nu_0(J_1-T^*J_2T)=\nu_0(J_2-TJ_1T^*).
\]
\end{lemma}

The next result contains two general factorization results:
assertion (i) contains an extension of the well-known Douglas
factorization, see \cite{Douglas,FW}, and assertion (ii) is a
specification of the so-called Bogn\'ar-Kr\'amli factorization, see
\cite{Bognarbook}: $A=B^*J_2B$ holds for some bounded operator $B$
if and only if $\nu_\pm(J_2)\ge \nu_\pm(A)$.

\begin{proposition}\label{BKcor1}
Let $A$, $B$, and $J_2$ be as in Theorem~\ref{thmB}, and assume that
$\nu_-(A)=\nu_-(J_2)<\infty$. Then:
\begin{enumerate}\def\labelenumi{\rm(\roman{enumi})}
\item The inequality
\begin{equation}\label{min2}
 A\ge B^*J_2 B
\end{equation}
holds if and only if $B=C|A|^{1/2}$ for some $J$-contractive operator $C\in[\sH_A,\sH_2]$;
in this case $C$ is unique and, in addition, $J$-bicontractive, i.e.,
$J_A-C^*J_2 C\ge 0$ and $J_2-CJ_A C^*\ge 0$.

\item The equality
\begin{equation}\label{min3}
 A = B^*J_2 B
\end{equation}
holds if and only if $B=C|A|^{1/2}$ for some $J$-isometric operator $C\in[\sH_A,\sH_2]$;
again $C$ is unique. In addition, $C$ is unitary if and only if $\ran B$ is dense in $\sH_2$.
\end{enumerate}
\end{proposition}
\begin{proof}
(i) The inequality \eqref{min2} means that $\nu_-(A - B^*J_2 B)=0$.
Hence the assumption $\nu_-(A)=\nu_-(J_2)<\infty$ implies the
equality \eqref{min}. Therefore, the desired factorization for $B$
is obtained from Theorem \ref{thmB}. Conversely, if $B=C|A|^{1/2}$
for some $J$-contractive operator $C$ then \eqref{min} holds by
Theorem \ref{thmB} and the assumption $\nu_-(A)=\nu_-(J_2)<\infty$
implies that $\nu_-(A - B^*J_2 B)=0$.

The fact that $C$ is actually $J$-bicontractive follows directly
from Lemma~\ref{inertia}.

(ii) Assume that \eqref{min3} holds. Then by part (i) it remain to
prove that in the factorization $B=C|A|^{1/2}$ the operator $C$ is
isometric. Substituting $B=C|A|^{1/2}$ into \eqref{min3} gives
\[
 A=|A|^{1/2}C^*J_2C|A|^{1/2}.
\]
Since $\dom C,\, \ran C^* \subset \cran A$ and
$A=|A|^{1/2}J_A|A|^{1/2}$, the previous identity implies the
equality $J_A=C^*J_2C$, i.e., $C$ is $J$-isometric. Conversely, if
$C$ is $J$-isometric then clearly \eqref{min3} holds.

Since $B=C|A|^{1/2}$ and $C\in[\sH_A,\sH_2]$, it is clear that $B$
has dense range in $\sH_2$ precisely when the range of $C$ is dense
in $\sH_2$. The (Kre\u{\i}n space) adjoint is a bounded operator
with $\dom C^{[*]}=\sH_2$. By isometry one has $C^{-1}\subset
C^{[*]}$, and thus $C^{-1}$ is also bounded, densely defined and
closed. Thus, the equality $C^{-1}=C^{[*]}$ prevails, i.e., $C$ is
$J$-unitary. Conversely, if $C$ is unitary then $C^{-1}=C^{[*]}$
holds and $\ran C=\dom C^{[*]}=\sH_2$. Consequently, $\ran B=\ran
C|A|^{1/2}$ is dense in $\sH_2$.
\end{proof}

If, in particular, $\nu_-(A)=\nu_-(J_2)=0$ then $0\le A\le B^*B$ and
Proposition \ref{BKcor1} combined with Theorem~\ref{T:1} yields the factorization and
range inclusion results proved in \cite[Theorem~1]{Douglas} with $A$ replaced by $A^*A$.
In particular, notice that if $\ran B^*\subset \ran |A|^{1/2}$, then already Theorem~\ref{T:1}
alone implies that $S=|A|^{[-1/2]}B^*$ is bounded and hence
$B^*B=|A|^{1/2}SS^*|A|^{1/2}\le \|S\|^2 A$.

Assertions in part (ii) of Corollary \ref{BKcor1} can be found in
the literature with a different proof. In fact, the first statement
in (ii) appears in \cite[Proposition~2.1, Corollary~2.6]{ACG87} while
the second statement in (ii) is proved in
\cite[Corollary~1.3]{CG89}. Another extension for Douglas'
factorization result can be found from \cite{Rod07}.

For a general treatment of isometric (not necessarily densely
defined) operators and isometric relations appearing in the proof of
Proposition \ref{BKcor1} the reader is referred to \cite{AIbook},
\cite[Section~2]{DHMS2006}, and \cite{DHMS2012}.

A slightly different viewpoint to Proposition~\ref{BKcor1} gives
the following statement, which can be viewed as an extension of a theorem by Shmul'yan \cite{S67}
on the factorization of bicontractions on Kre\u{\i}n spaces;
for a related abstract Leech theorem, see \cite[Section~3.4]{DritRov90}.

\begin{corollary}\label{BKcor2}
Let $A\in[\sH_1]$ be selfadjoint, let $B\in[\sH_1,\sH_2]$, and let
$J_2=J_2^*=J_2^{-1}\in[\sH_2]$ with $\nu_-(J_2)<\infty$. Then:
\begin{enumerate}\def\labelenumi{\rm(\roman{enumi})}
\item
\[
 A\ge B^*J_2 B  \quad\text{and}\quad  \nu_-(A) = \nu_-(J_2)
\]
if and only if $B=C|A|^{1/2}$ for some $J$-bicontractive
operator $C\in[\sH_A,\sH_2]$; in this case $C$ is unique.

\item
\[
 A=B^*J_2 B  \quad\text{and}\quad  \nu_-(A) = \nu_-(J_2)
\]
if and only if $B=C|A|^{1/2}$ for some $J$-bicontractive
operator $C$ which is also $J$-isometric, i.e., $J_A-C^*J_2 C=
0$ and $J_2-CJ_A C^*\ge 0$; again $C$ is unique.
\end{enumerate}
\end{corollary}
\begin{proof}
Observe that if $C$ is $J$-bicontractive, then an application of
Lemma~\ref{inertia} shows that $\nu_-(J_2)=\nu_-(J_A)=\nu_-(A)$. Now
the stated equivalences can be obtained from
Proposition~\ref{BKcor1}.
\end{proof}

This section is finished with an extension of the classical Sylvester's criterion,
that is actually obtained as a consequence of Theorem~\ref{T:1}.

\begin{proposition}\label{sylvester}
Let $A=(A_{ij})_{i,j=1}^{2}$ be an arbitrary selfadjoint block operator in $\sH=\sH_1\oplus\sH_2$,
which satisfies the range inclusion \eqref{E:1}, and let $S=|A_{11}|^{[-1/2]}A_{12}$.
Then $\nu_-(A)<\infty$ if and only if
$\nu_-(A_{11})<\infty$ and $\nu_-(A_{22}-S^*JS)<\infty$; in this case
\[
 \nu_-(A)=\nu_-(A_{11})+\nu_-(A_{22}-S^*JS).
\]
In particular, $A\ge 0$ if and only if $\ran A_{12}\subset\ran |A_{11}|^{1/2}$,
$A_{11}\ge 0$, and $A_{22}-S^*JS\ge 0$.
\end{proposition}
\begin{proof}
By the assumption \eqref{E:1} $S=|A_{11}|^{[-1/2]}A_{12}$ is an everywhere defined bounded
operator and, since $A_{11}=|A_{11}|^{1/2}J|A_{11}|^{1/2}$ (cf. Theorem~\ref{T:1}), the following equality holds:
\[
 A=\begin{pmatrix}
   |A_{11}|^{1/2} &0\\
   S^*J&I
  \end{pmatrix}
  \begin{pmatrix}
   J & 0\\
   0& A_{22}-S^*JS
  \end{pmatrix}
  \begin{pmatrix}
   |A_{11}|^{1/2} & JS \\
   0& I
  \end{pmatrix},
\]
i.e. $A=B^*EB$ where $E$ stands for the diagonal operator with
$\nu_-(E)=\nu_-(A_{11})+\nu_-(A_{22}-S^*JS)$
and the triangular operator $B$ on the right side is bounded and has dense range in
$\cran A_{11}\oplus\sH_2$.
Clearly, $\nu_-(A)\le \nu_-(E)$ and it remains to prove that if $\nu_-(A)<\infty$
then $\nu_-(A)=\nu_-(E)$.

To see this assume that $\nu_-(A)<\nu_-(E)$. We claim that $\ran B$ contains an $E$-negative subspace
$\sL$ with dimension $\dim \sL>\nu_-(A)$. Assume the converse and let $\sL\subset \ran B$ be a maximal $E$-negative subspace
with $\dim \sL\le \nu_-(A)$. Then $(E\sL)^\perp$ must be $E$-nonnegative, since if $v\perp E\sL$ and
$(Ev,v)<0$, then  $\spn\{v+\sL\}$ would be a proper $E$-negative extension of $\sL$.
Since $E\sL$ is finite-dimensional and $\ran B$ is dense in $\cran A_{11}\oplus\sH_2$, $\ran B$ has dense intersection
with $(\cran A_{11}\oplus\sH_2)\ominus E\sL$, and hence the closure of this subspace is also $E$-nonnegative.
Consequently, $\nu_-(E)=\nu_-(\sL)$, a contradiction with the assumption $\nu_-(E)>\nu_-(A)$.
This proves the claim that $\ran B$ contains an $E$-negative subspace
$\sL$ with $\dim \sL>\nu_-(A)$. However, then the subspace $\sL'=\{u\in \cran A_{11}\oplus\sH_2:Bu\in\sL\}$
satisfies $\dim \sL'\ge \dim \sL$ and, moreover, $\sL'$ is $A$-negative: $(Au,u)=(EBu,Bu)<0$, $u\in\sL'$, $u\neq 0$.
Thus, $\nu_-(A)\ge \dim \sL$, a contradiction with $\dim \sL>\nu_-(A)$. This completes the proof.
\end{proof}

Proposition~\ref{sylvester} completes Theorem~\ref{T:1}: it shows
that if $\ran A_{12}\subset\ran |A_{11}|^{1/2}$ then
$A_{11}=J|A_{11}|$ and $A_{12}=|A_{11}|^{1/2}S$ imply that
$A_{21}|A_{11}|^{[-1/2]}J|A_{11}|^{[-1/2]}A_{12}=S^*JS$. Hence the
negative index of $A$ can be calculated by using the following
version of a \textit{generalized of Schur complement}:
\begin{equation}\label{genSchur}
 \nu_-(A)=\nu_-(A_{11})+\nu_-(A_{22}-A_{21}|A_{11}|^{[-1/2]}J|A_{11}|^{[-1/2]}A_{12}).
\end{equation}
The addition made in Proposition~\ref{sylvester} concerns
selfadjoint operators $A_{22}$ that are not solutions to the
original completion problem for $A^0$.

\section{Lifting of operators with finite negative index}
\label{sec3}

As a first application of the completion problem solved in
Section~\ref{sec1} it is shown how nicely some lifting results
established in a series of papers by Arsene, Constantinescu, and
Gheondea, see \cite{AG82,ACG87,CG89,CG92}, as well as in Dritschel,
see \cite{Drit90,DritRov90} (see also further references appearing
in these papers), on contractive operators with finite number of
negative squares can be derived from Theorem~\ref{T:1}.

For this purpose some standard notations are now introduced. Let
$(\sH_1,(\cdot,\cdot)_{1})$ and $(\sH_2,(\cdot,\cdot)_{2})$ be
Hilbert spaces and let $J_1$ and $J_2$ be symmetries in $\sH_1$ and
$\sH_2$, i.e. $J_i=J_i^*=J_i^{-1}$, so that
$(\sH_i,(J_i\cdot,\cdot)_{i})$, $i=1,2$, becomes a Kre\u{\i}n space.
Then associate with $T\in[\sH_1,\sH_2]$ the corresponding defect and
signature operators
\[
 D_T=|J_1-T^*J_2T|^{1/2},\quad J_T=\sign(J_1-T^*J_2T), \quad \sD_T=\cran D_T,
\]
where the so-called defect subspace $\sD_T$ can be considered as a
Kre\u{\i}n space with the fundamental symmetry $J_T$. Similar
notations are used with $T^*$:
\[
 D_{T^*}=|J_2-TJ_1T^*|^{1/2},\quad J_{T^*}=\sign(J_2-TJ_1T^*), \quad \sD_{T^*}=\cran D_{T^*}.
\]
By definition $J_TD_T^2=J_1-T^*J_2T$ and $J_TD_T=D_TJ_T$ with
analogous identities for $D_{T^*}$ and $J_{T^*}$. In addition,
\begin{equation}\label{eqC3}
 (J_1-T^*J_2T)J_1T^*=T^*J_2(J_2-TJ_1T^*), \,
 (J_2-TJ_1T^*)J_2T=TJ_1(J_1-T^*J_2T).
\end{equation}

Recall that $T\in[\sH_1,\sH_2]$ is said to be a $J$-contraction if $J_1-T^*J_2T\ge 0$,
i.e. $\nu_-(J_1-T^*J_2T)=0$. If, in addition, $T^*$ is a $J$-contraction, $T$ is termed as
a $J$-bicontraction, in which case $\nu_-(J_1)=\nu_-(J_2)$ by Lemma~\ref{inertia}.
In what follows it is assumed that
\[
 \kappa_1:=\nu_-(J_1-T^*J_2T)<\infty,\quad  \kappa_2:=\nu_-(J_2-TJ_1T^*)<\infty.
\]
In this case Lemma~\ref{inertia} shows that
\begin{equation}\label{iner01}
 \nu_-(J_2) = \nu_-(J_1) + \kappa_2-\kappa_1.
\end{equation}

The aim in this section is to show applicability of
Theorem~\ref{T:1} in establishing formulas for so-called liftings
$\wt T$ of $T$ with prescribed negative indices $\wt\kappa_1$ and
$\wt\kappa_2$ for the defect subspaces. Given a bounded operator
$T\in[\sH_1,\sH_2]$ the problem is to describe all operators $\wt T$
from the extended Kre\u{\i}n space
$(\sH_1\oplus\sH_1^\prime,J_1\oplus J_1^\prime)$ to the extended
Kre\u{\i}n space $(\sH_2\oplus\sH_2^\prime,J_2\oplus J_2^\prime)$
such that
\[
\textbf{$(*)$} \qquad P_2 \wt T\uphar \sH_1 = T  \quad\text{and}\quad
 \nu_-(\wt J_1-\wt T^*\wt J_2 \wt T)=\wt\kappa_1, \quad
 \nu_-(\wt J_2-\wt T\wt J_1 \wt T^*)=\wt\kappa_2,
\]
with some fixed values of $\wt\kappa_1,\wt\kappa_2<\infty$. Here
$P_i$ stands for the orthogonal projection from
$\wt\sH_i=\sH_i\oplus\sH_i^\prime$ onto $\sH_i$ and $\wt
J_i=J_i\oplus J_i^\prime$, $i=1,2$. In addition, it is assumed that
the exit spaces are Pontryagin spaces, i.e., that
\[
 \nu_-(J_1^\prime), \nu_-(J_2^\prime) <\infty.
\]
Following \cite{ACG87,CG89} consider first the following column extension problem:

$\textbf{$(*)_c$}$ Give a description of all (column) operators
$T_c=\col\begin{pmatrix}T & C \end{pmatrix}\in
[\sH_1,\sH_2\oplus\sH_2^\prime]$, such that
$\nu_-(J_1-T_c^*\wt J_2T_c)=\wt \kappa_1\,(<\infty)$.

Since $J_1-T_c^*\wt J_2T_c=J_1-T^*J_2T-C^*J_2^\prime C$, then
necessarily (see Section~\ref{sec2})
\[
 \wt\kappa_1 \ge \kappa_1-\nu_-(C^*J_2^\prime C)\ge \kappa_1 -\nu_-(J_2^\prime).
\]
Moreover, it is clear that $\wt\kappa_2\ge \kappa_2$, since
$J_2-TJ_1T^*$ appears as the first diagonal entry of the $2\times 2$
block operator $\wt J_2-T_c J_1T_c^*$ when decomposed w.r.t.
$\wt\sH_i=\sH_i\oplus\sH_i^\prime$, $i=1,2$.

With the minimal value of $\wt \kappa_1$ all solutions to this
problem will now be described by applying Theorem~\ref{T:1} to an
associated $2\times 2$ block operator $T_C$ appearing in the proof
below; in fact the result is just a special case of
Theorem~\ref{thmB}.

\begin{lemma}\label{corL1}
Let $\wt\kappa_1=\nu_-(J_1-T_c^*\wt J_2T_c)$ and assume that
$\wt\kappa_1=\kappa_1-\nu_-(J_2^\prime)(\ge 0)$. Then $\ran
C^*\subset \ran D_T$ and the formula
\[
 T_c=\begin{pmatrix}T \\ K^*D_{T} \end{pmatrix}
\]
establishes a one-to-one correspondence between the set of all
solutions to Problem $\textbf{$(*)_c$}$ and the set of all
$J$-contractions $K\in[\sH'_2,\sD_{T}]$.
\end{lemma}
\begin{proof}
To make the argument more explicit consider the following block
operator
\[
 T_C:=\begin{pmatrix} J_1-T^*J_2T & C^* \\ C & J_2^\prime \end{pmatrix}
 = \begin{pmatrix} I & C^*J_2^\prime \\ 0 & I \end{pmatrix}
  \begin{pmatrix} J_1-T_c^*\wt J_2T_c & 0 \\ 0 & J_2^\prime \end{pmatrix}
   \begin{pmatrix} I  & 0 \\ J_2^\prime C & I \end{pmatrix}.
\]
Clearly $\nu_-(T_C)=\nu_-(J_1-T_c^*\wt
J_2T_c)+\nu_-(J_2^\prime)<\infty$, which combined with
$\wt\kappa_1=\kappa_1-\nu_-(J_2^\prime)$ shows that
$\nu_-(T_C)=\kappa_1=\nu_-(J_1-T^*J_2T)$. Now, the statement is
obtained from Theorem~\ref{T:1} or, more directly, just by applying
Theorem~\ref{thmB}.
\end{proof}

\begin{remark}
(i) The above proof, which essentially makes use of an associated
$2\times 2$ block operator $T_c$ (being a special case of the block
operator $H$ in \eqref{H} behind Theorem~\ref{thmB}), is new even in
the case of Hilbert space contractions. In particular, it shows that
the operator $K$ in Lemma~\ref{corL1} coincides with the operator
$S$ that gives the minimal solution $S^*J_{T}S$ to the completion
problem associated with $T_C$; the $J$-contractivity of $K$ itself
is equivalent to the fact that $T_C$ is also a solution precisely
when $\wt\kappa=\kappa-\nu_-(J_2^\prime)$.

(ii) The existence of a solution to Problem $\textbf{$(*)_c$}$ is proved here using only the condition $\wt\kappa_1=\kappa_1-\nu_-(J_2^\prime)\,(\ge 0)$. The corresponding result in \cite[Lemma~2.2]{CG89} is formulated (and formally also proved) under the additional condition $\wt\kappa_2=\kappa_2$.
In the case that $\nu_-(J_1)<\infty$ the equality $\wt\kappa_2=\kappa_2$ follows automatically
from the equality $\wt\kappa_1=\kappa_1-\nu_-(J_2^\prime)$: to see this apply
\eqref{iner01} to $T$ and $T_c$, which leads to $\nu_-(J_1)+\kappa_2=\nu_-(J_1)+\wt\kappa_2$,
so that $\nu_-(J_1)<\infty$ implies $\kappa_2=\wt\kappa_2$.
Naturally, in Lemma~\ref{corL1} the condition $\wt\kappa_2=\kappa_2$ follows from the
condition $\wt\kappa_1=\kappa_1-\nu_-(J_2^\prime)$ also in the case where $\nu_-(J_1)=\infty$;
see Corollary~\ref{k2cor} below.

Finally, it is mentioned that for a Pontryagin space operator
$T$ the result in Lemma~\ref{corL1} was proved in \cite[Lemma~5.2]{ACG87}.
\end{remark}

In a dual manner we can treat the following row extension problem;
again initially considered in \cite{ACG87,CG89}:

$\textbf{$(*)_r$}$ Give a description of all operators
$T_r=\begin{pmatrix}T & R\end{pmatrix}\in
[\sH_1\oplus\sH'_1,\sH_2]$, such that
$\nu_-(\wt J_1-T_r^* J_2 T_r)=\wt\kappa_2\,(<\infty)$.

Analogous to the case of column operators, $J_2-T_r\wt J_1T_r^*=J_2-TJ_1T^*-RJ_1^\prime R^*$
gives the estimate
\[
 \wt\kappa_2 \ge \kappa_2-\nu_-(RJ_1^\prime R^*)
 \ge \kappa_2-\nu_-(J_1^\prime).
\]
Moreover, it is clear that $\wt\kappa_1\ge \kappa_1$.
With the minimal value of $\wt \kappa_2$ all solutions to
Problem $\textbf{$(*)_r$}$ are established by applying Theorem~\ref{T:1}
to an associated $2\times 2$ block operator $T_R$.

\begin{lemma}\label{corL2}
Let $\wt\kappa_2=\nu_-(J_2-T_r\wt J_1T_r^*)$ and assume that
$\wt\kappa_2=\kappa_2 - \nu_-(J_1^\prime)(\ge 0)$. Then $\ran
R\subset \ran D_{T^*}$ and the formula
\[
 T_r=\begin{pmatrix}T & D_{T^*} B\end{pmatrix}
\]
establishes a one-to-one correspondence between the set of all
solutions to Problem $\textbf{$(*)_r$}$ and the set of all
$J$-contractions $B\in[\sH'_1,\sD_{T^*}]$.
\end{lemma}
\begin{proof}
To prove the statement via Theorem~\ref{thmB} (or Theorem~\ref{T:1})
consider
\[
 T_R:=\begin{pmatrix} J_2-TJ_1T^* & R \\ R^* & J_1^\prime \end{pmatrix}
 = \begin{pmatrix} I & RJ_1^\prime \\ 0 & I \end{pmatrix}
  \begin{pmatrix} J_2-T_r\wt J_1T_r^* & 0 \\ 0 & J_1^\prime \end{pmatrix}
   \begin{pmatrix} I  & 0 \\ J_1^\prime R^* & I \end{pmatrix}.
\]
Then clearly $\nu_-(T_R)=\nu_-(J_2-T_r\wt
J_1T_r^*)+\nu_-(J_1^\prime)$ and hence the assumption
$\wt\kappa_2=\kappa_2 - \nu_-(J_1^\prime)$ is equivalent to
$\nu_-(T_R)=\kappa=\nu_-(J_2-TJ_1T^*)$. Hence, again the statement
follows from Theorem~\ref{thmB}.
\end{proof}

Remarks similar to those made after Lemma~\ref{corL1} can be done here, too.
In particular, the corresponding result in \cite[Lemma~2.1]{CG89} is formulated under the additional condition $\wt\kappa_1=\kappa_1$: here this equality will be a consequence from the equality
$\wt\kappa_2=\kappa_2 - \nu_-(J_1^\prime)$; cf. Corollary~\ref{k2cor} below.

To prove the main result concerning parametrization of all $2\times
2$ liftings in a larger Kre\u{\i}n space with minimal signature for
the defect operators an indefinite version of the commutation
relation of the form $TD_T=D_{T^*}T$ is needed; these involve
so-called link operators introduced in \cite[Section 4]{ACG87}.

We will give a simple proof for the construction of link operators
(see \cite[Proposition 4.1]{ACG87}) by applying Heinz inequality
combined with the basic factorization result from \cite{Douglas}.
The first step is formulated in the next lemma, which is connected
to a result of M.G. Kre\u{\i}n \cite{Kr47b} concerning continuity of
a bounded Banach space operator which is symmetric w.r.t. to a
continuous definite inner product; the existence of link was proved
proved in \cite{ACG87} via this result of Kre\u{\i}n. Here a statement,
analogous to that of Kre\u{\i}n, is formulated in pure Hilbert space operator
language by using the modulus of the product operator;
see \cite[Lemma B2]{DritRov90},
where Kre\u{\i}n's result is presented with a proof due to W.T. Reid.

\begin{lemma}\label{2norms}
Let $S\in[\sH_1,\sH_2]$ and let $H\in[\sH_2]$ be nonnegative. Then
\[
  HS=(HS)^* \quad\Rightarrow\quad |HS| \le \mu H \text{ for some } \mu <\infty.
\]
\end{lemma}
\begin{proof}
Since $HS$ is selfadjoint, one obtains
\[
 (HS)^2=HSS^*H \le \mu^2 H^2, \quad \mu = \|S\|<\infty.
\]
Now by Heinz inequality (see e.g. \cite[Theorem~10.4.2]{BirSol87})
we get
\[
 |HS|=(HSS^*H)^{1/2} \le \mu H. \qedhere
\]
\end{proof}

\begin{corollary}\label{linkoper}
Let $T\in[\sH_1,\sH_2]$ and let $J_1$ and $J_2$ be symmetries in
$\sH_1$ and $\sH_2$ as above. Then there exist unique operators
$L_T\in [\sD_T,\sD_{T^*}]$ and $L_{T^*}\in [\sD_{T^*},\sD_T]$ such
that
\[
 D_{T^*}L_T=TJ_1 D_T\uphar \sD_T, \quad
 D_T L_{T^*}=T^*J_2 D_{T^*}\uphar \sD_{T^*};
\]
in fact, $L_T=D_{T^*}^{[-1]}TJ_1 D_T\uphar \sD_T$ and
$L_{T^*}=D_T^{[-1]}T^*J_2 D_{T^*}\uphar \sD_{T^*}$.
\end{corollary}
\begin{proof}
Denote $S=J_{T^*}J_2TJ_TJ_1T^*$. Then \eqref{eqC3} implies that
\[
 D_{T^*}^2 S=(J_2-TJ_1T^*)J_2TJ_TJ_1T^*
  =TJ_1(J_1-T^*J_2T)J_TJ_1T^*
  =TJ_1D_T^2J_1T^*\ge 0,
\]
so that $D_{T^*}^2 S$ is nonnegative and, in particular,
selfadjoint. By Lemma~\ref{2norms} with $\mu=\|S\|$ one has
\[
 0\le TJ_1D_T^2J_1T^*=D_{T^*}^2 S \le \mu D_{T^*}^2.
\]
This last inequality is equivalent to the factorization
$TJ_1D_T\uphar \sD_T=D_{T^*} L_T$ with a unique operator $L_T\in
[\sD_T,\sD_{T^*}]$, see \cite[Theorem~1]{Douglas}, which by means of
Moore-Penrose generalized inverse can be rewritten as indicated.

The second formula is obtained by applying the first one to $T^*$.
\end{proof}

The following identities can be obtained with direct calculations;
see \cite[Section~4]{ACG87}:
\begin{equation}\label{Link2}
\begin{array}{c}
  L_T^* J_{T^*}\uphar \sD_{T^*}=J_{T}L_{T^*};\\
  (J_T-D_TJ_1D_T)\uphar \sD_T=L_T^*J_{T^*}L_T;\\
  (J_{T^*}-D_{T^*}J_2D_{T^*})\uphar \sD_{T^*}=L_{T^*}^* J_T L_{T^*}.
\end{array}
\end{equation}

The next corollary contains the promised identity $\wt\kappa_1=\kappa_1$
under the assumption $\wt\kappa_2=\kappa_2-\nu_-(J_2^\prime)\ge 0$ in Lemma~\ref{corL1}.
Similarly $\wt\kappa_1=\kappa_1-\nu_-(J_1^\prime)$ implies $\wt\kappa_2=\kappa_2$;
the general result for the first case can be formulated as follows (and there is similar
result for the latter case).

\begin{corollary}\label{k2cor}
Let $R$ be a bounded operator such that $\ran R\subset \ran D_{T^*}$ and
let $T_r$ be the corresponding row operator and denote $\wt\kappa_1=\nu_-(\wt J_1-T_r^*J_2T_r)$.
Then $R=D_{T^*}B$ for a (unique) bounded operator $B\in[\sH'_1,\sD_{T^*}]$ and
\[
 \wt\kappa_1=\kappa_1+\nu_-(J_1^\prime- B^*J_{T^*} B).
\]
In particular, $J$-contractivity of $B$ is equivalent to $\wt\kappa_1=\kappa_1$.
\end{corollary}
\begin{proof}
Recall that $\ran R\subset \ran D_{T^*}$ is equivalent to the
factorization $R=D_{T^*}B$. By applying the commutation relations in
Corollary~\ref{linkoper} together with the identities \eqref{Link2}
one gets the following expression for $ J_{T_r}D_{T_r}^2$:
\begin{equation}\label{Jtr}
\begin{array}{rl}
 J_{T_r}D_{T_r}^2  &
 = \begin{pmatrix} J_1-T^*J_2T & -T^*J_2D_{T^*}B  \\
  -B^*D_{T^*} J_2 T & J_1^\prime - B^*D_{T^*} J_2 D_{T^*}B\end{pmatrix} \\ &
  = \begin{pmatrix} J_TD_T^2 & -D_T L_{T^*}B  \\
  -B^*L^*_{T^*} D_T & J_BD_{B}^2 + B^*L^*_{T^*} J_T L_{T^*}B\end{pmatrix}.
\end{array}
\end{equation}
Now apply Proposition~\ref{sylvester} and calculate the Schur
complement, cf. \eqref{genSchur},
\[
 J_BD_{B}^2 + B^*L^*_{T^*} J_T L_{T^*}B
 -B^*L^*_{T^*} D_T (D_T^{[-1]}J_TD_T^{[-1]}) D_T L_{T^*}B
 =J_BD_{B}^2,
\]
to see that $\wt\kappa_1=\nu_-(J_1-T^*J_2T)+\nu_-(J_1^\prime- B^*J_{T^*} B)$.
\end{proof}

By means of Lemmas~\ref{corL1},~\ref{corL2} and the link operators
in Corollary~\ref{linkoper} one can now establish the main result
concerning the lifting problem $(*)$.

First notice that if Problem $(*)$ has a solution, then by treating
$\wt T$ as a row extension of its first column $T_c$ and as a column
extension of its first row $T_r$ one gets from the inequalities
preceding Lemmas~\ref{corL1},~\ref{corL2} the estimates
\begin{equation}\label{iner03}
\begin{array}{l}
 \wt\kappa_1\ge \kappa_1(T_r)-\nu_-(J_2^\prime)\ge \kappa_1-\nu_-(J_2^\prime); \\
 \wt\kappa_2\ge \kappa_2(T_c)-\nu_-(J_1^\prime)\ge \kappa_2-\nu_-(J_1^\prime).
\end{array}
\end{equation}
Under the minimal choice of the indices $\wt\kappa_1$ and
$\wt\kappa_2$ Problem $(*)$ is already solvable; all
solutions are described by the following result, which was initially
proved in \cite[Theorem~2.3]{CG89} with the aid of \cite[Theorem~5.3]{ACG87}.
Here a different proof is presented, again based on an application of Theorem~\ref{T:1}.

\begin{theorem}\label{Lifthm}
Let $\wt T$ be a bounded operator from
$(\sH_1\oplus\sH_1^\prime,J_1\oplus J_1^\prime)$ to
$(\sH_2\oplus\sH_2^\prime,J_2\oplus J_2^\prime)$ such that $P_2 \wt
T\uphar \sH_1 = T$. Assume that $0\leq \kappa_1 -
\nu_-(J_2^\prime)=\wt\kappa_1<\infty$ and $0\leq \kappa_2 -
\nu_-(J_1^\prime)=\wt\kappa_2<\infty$. Then the Problem $(*)$ is
solvable and the formula
\[
 \wt T=\begin{pmatrix}T & D_{T^*} \Gamma_1 \\
 \Gamma_2 D_T & -\Gamma_2 L_T^* J_{T^*} \Gamma_1 + D_{\Gamma_2^*}\Gamma D_{\Gamma_1}\end{pmatrix}
\]
establishes a one-to-one correspondence between the set of all
solutions to Problem $(*)$ and the set of triplets
$\{\Gamma_1,\Gamma_2,\Gamma\}$ where
$\Gamma_1\in[\sH'_1,\sD_{T^*}]$ and $\Gamma_2^*\in[\sH'_2,\sD_T]$ are $J$-contractions
and $\Gamma\in[\sD_{\Gamma_1},\sD_{\Gamma_2^*}]$ is a Hilbert space contraction.
\end{theorem}
\begin{proof}
Assume that there is a solution $\wt T$ to Problem $(*)$ and write it in the form
\[
 \wt T=\begin{pmatrix} T & R \\ C & X \end{pmatrix}
\]
with the first column denoted by $T_c$ and first row denoted by
$T_r$, and assume that $\wt\kappa_1=\kappa_1 - \nu_-(J_2^\prime)$
and $\wt\kappa_2= \kappa_2 - \nu_-(J_1^\prime)$. Then \eqref{iner03}
shows that $\kappa_1=\kappa_1(T_r)$ and $\kappa_2=\kappa_2(T_c)$.
Hence Lemma~\ref{corL2} can be applied by viewing $\wt T$ as a row
extension of $T_c$ to get a range inclusion and then from
Corollary~\ref{k2cor} one gets the equality
$\wt\kappa_1=\kappa_1(T_c)$. Similarly applying Lemma~\ref{corL1}
and the analog of Corollary~\ref{k2cor} to column operator $\wt T$
one gets the equality $\wt\kappa_2=\kappa_2(T_r)$. Thus
$\kappa_1(T_c)=\kappa_1 - \nu_-(J_2^\prime)$ and
$\kappa_2(T_r)=\kappa_2 -\nu_-(J_1^\prime)$. Consequently, one can
apply Lemma~\ref{corL1} to the first column $T_c$ and
Lemma~\ref{corL2} to the first row $T_r$ to get the stated
factorizations $C=\Gamma_2 D_T$ and $R=D_{T^*}\Gamma_1$ with unique
$J$-contractions $\Gamma_1$ and $\Gamma_2^*$.

To establish a formula for $X$ we proceed by considering the block
operator
\[
 H:=\begin{pmatrix} J_{T_r}D_{T_r}^2 & T_{r,2}^* \\ T_{r,2} & J_2^\prime \end{pmatrix},
\]
where $T_{r,2}$ denotes the second row of $\wt T$. It is
straightforward to derive the following formula for the Schur
complement
\[
 J_{T_r}D_{T_r}^2-T_{r,2}^* J_2^\prime T_{r,2}=\wt J_1-\wt T^*\wt J_2 \wt T.
\]
Thus
$\nu_-(H)=\wt\kappa_1+\nu_-(J_2^\prime)=\kappa_1=\nu_-(J_{T_r})$ and
one can apply Theorem~\ref{T:1} to get the factorization
$T_{r,2}^*=D_{T_r} \wt K$ with a unique $\wt
K\in[\sH_2^\prime,\sD_{T_r}]$ satisfying $\wt K^* J_{T_r} \wt K \le
J_2^\prime$, i.e., $\wt K$ is a $J$-contraction; see
Theorem~\ref{thmB}.

It follows from \eqref{Jtr} that
\[
 J_{T_r}D_{T_r}^2
 = \begin{pmatrix} D_T & 0 \\ -\Gamma_1^* L^*_{T^*}J_T  & D_{\Gamma_1}\end{pmatrix}
     \begin{pmatrix} J_T & 0 \\ 0 & I_{D_{\Gamma_1}} \end{pmatrix}
     \begin{pmatrix} D_T & -J_T L_{T^*}\Gamma_1 \\ 0  & D_{\Gamma_1}\end{pmatrix}
     =:B^*\wh J B.
\]
Since here $\nu_-(J_{T_r})=\kappa_1=\nu_-(J_{T})$ and $B$ is a
triangular operator whose range is dense in $\sD_T\oplus
\sD_{\Gamma_1}$ (the diagonal entries $D_T$ and $D_{\Gamma_1}$ of
$B$ have dense ranges by definition), there is a unique Pontryagin
space $J$-unitary operator $U$ from $\sD_{T_r}$ onto $\sD_T\oplus
\sD_{\Gamma_1}$ such that $B=UD_{T_r}$; see Proposition~\ref{BKcor1}
(ii). It follows that $K^*:=(U^{-1})^*\wt K$ is a $J$-contraction
from $\sH_2^\prime$ to $\sD_T\oplus \sD_{\Gamma_1}$ and $KB=\wt K^*
D_{T_r}=T_{r,2}$. Now $J_2^\prime-K\wh J K^*\ge 0$ gives
\begin{equation}\label{eqKK}
 0\le K_1K_1^*\le J_2^\prime-K_0 J_T K_0^*, 
\end{equation}
where $K=(K_0\,K_1)$ is considered as a row operator, and
$T_{r,2}=KB$ reads as
\[
 \Gamma_2D_{T}=K_0 D_T, \quad X=-K_0 J_T L_{T^*}\Gamma_1 + K_1 D_{\Gamma_1}.
\]
Since all contractions that are involved are unique, $K_0=\Gamma_2$,
$J_2^\prime-K_0 J_T K_0^* = D_{\Gamma_2^*}^2$, and \eqref{eqKK}
implies that there is a unique Hilbert space contraction
$\Gamma\in[\sD_{\Gamma_1},\sD_{\Gamma_2^*}]$ such that
$K_1=D_{\Gamma_2^*}\Gamma$. The desired formula for $\wt T$ is
proven (cf. \eqref{Link2}). It is clear from the proof that every
operator $\wt T$ of the stated form is a solution and that there is
one-to-one correspondence via the triplets
$\{\Gamma_1,\Gamma_2,\Gamma\}$ of $J$-contractions.
\end{proof}

\begin{remark}
(i) By replacing $\wt T$ with its adjoint $\wt T^*$ it is clear that
all formulas remain the same and are obtained by changing $T$ with
$T^*$ and interchanging the roles of the indices $1$ and $2$; see
also \eqref{Link2}. This connects the considerations with row and
column operators to each other.

(ii) If $\kappa_1=0$ so that $J_1-T^*J_2T\ge 0$, then the above
proof becomes slightly simpler since then $J_{T_r}$, $J_T$, and
$J_2^\prime$ are identity operators and $\wt K$ is a Hilbert space
contraction. Then Theorem~\ref{Lifthm} gives all contractive
liftings of a contraction in a Kre\u{\i}n space. If in addition
$\kappa_2=0$, then one gets all bicontractive liftings of a
bicontraction in a Kre\u{\i}n space with Pontryagin spaces as exit
spaces. In the case special case that the exit spaces are Hilbert
spaces ($\nu_-(J_1)=\nu_-(J_2)=0$ and $\kappa_1=\kappa_2=0$)
Theorem~\ref{Lifthm} coincides with \cite[Theorem~3.6]{Drit90}. In
fact, the present proof can be seen as a further development of the
proof appearing in that paper; see also further references and
historical remarks given in \cite{Drit90,DritRov90}.
\end{remark}

\section{Contractive extensions of contractions with minimal negative indices}
\label{sec4}

Let $\sH_1$ be a closed linear subspace of the Hilbert space $\sH$,
let $T_{11}=T_{11}^*\in[\sH_1]$ be an operator such that
$\nu_-(I-T_{11}^2)=\kappa<\infty$. Denote
\begin{equation}\label{JJpm}
 J=\sign(I-T_{11}^2), \quad
 J_+=\sign(I-T_{11}),\, \text{ and }\,
 J_-=\sign(I+T_{11}),
\end{equation}
and let $\kappa_+=\nu_-(I-T_{11})$ and $\kappa_-=\nu_-(I+T_{11})$.
It is obvious that $J=J_-J_+=J_+J_-$. Moreover, there is an equality
$\kappa=\kappa_- +\kappa_+$ as stated in the next lemma.

\begin{lemma}\label{sign}
Let $T=T^*\in[\sH_1]$ be an operator such that
$\nu_-(I-T^2)=\kappa<\infty$ then $\nu_-(I-T^2)=\nu_-(I+T)+\nu_-(I-T).$
\end{lemma}
\begin{proof}
Let $E_t(\cdot)$ be resolution of identity of $T$. Then by the spectral mapping theorem
the spectral subspace corresponding to the negative spectrum of $I-T^2$ is
given by $E_t((\infty;-1)\cup(1;\infty))=E_t((-\infty;-1))\oplus E_t((1;\infty))$.
Consequently, $\nu_-(I-T^2)=\dim E_t((-\infty;-1))+\dim E_t((1;\infty))=\nu_-(I+T)+\nu_-(I-T)$.
\end{proof}

The next problem concerns the existence and a description of
selfadjoint operators $T$ such that $\wt A_+=I+T$ and $\wt A_-=I-T$
solve the corresponding completion problems
\begin{equation}\label{E:A}
A_{\pm}^0=
\begin{pmatrix}
 I\pm T_{11}&\pm T_{21}^*\\
 \pm T_{21}&\ast
\end{pmatrix},
\end{equation}
under \emph{minimal index conditions} $\nu_-(I+T)=\nu_-(I+T_{11})$,
$\nu_-(I-T)=\nu_-(I-T_{11})$, respectively. Observe, that if $I\pm
T$ provides an arbitrary completion to $A_{\pm}^0$ then clearly
$\nu_-(I \pm T)\ge \nu_-(I \pm T_{11})$. Thus by Lemma~\ref{sign}
the two minimal index conditions above are equivalent to the single
condition $\nu_-(I-T^2)=\nu_-(I-T_{11}^2)$.

Unlike with the case of a selfadjoint contraction $T_{11}$, this
problem need not have solutions when $\nu_-(I-T_{11}^2)>0$. It is
clear from Theorem~\ref{T:1} that the conditions $\ran
T_{21}^*\subset \ran|I-T_{11}|^{1/2}$ and $\ran T_{21}^*\subset
\ran|I+T_{11}|^{1/2}$ are necessary for the existence of solutions;
however alone they are not sufficient.

The next theorem gives a general solvability criterion for the
completion problem \eqref{E:A} and describes all solutions to this
problem. As in the definite case, there are minimal solutions $A_+$
and $A_-$ which are connected to two extreme selfadjoint extensions
$T$ of
\begin{equation}\label{Tcol}
T_1=\begin{pmatrix} T_{11}\\ T_{21}
\end{pmatrix}:\sH_1\to\begin{pmatrix}\sH_1\\\sH_2\end{pmatrix},
\end{equation}
now with finite negative index $\nu_-(I-T^2)=\nu_-(I-T_{11}^2)>0$.
The set of all solutions $T$ to the problem \eqref{E:A} will be
denoted by $\Ext_{T_1,\kappa}(-1,1)$.

\begin{theorem}\label{T:contr}
Let $T_1$ be a symmetric operator as in \eqref{Tcol} with
$T_{11}=T_{11}^*\in[\sH_1]$ and $\nu_-(I-T_{11}^2)=\kappa<\infty$,
and let $J=\sign(I-T_{11}^2)$. Then the completion problem for
$A_{\pm}^0$ in \eqref{E:A} has a solution $I\pm T$ for some $T=T^*$
with $\nu_-(I-T^2)=\kappa$ if and only if the following condition is
satisfied:
\begin{equation}\label{crit}
 \nu_-(I-T_{11}^2)=\nu_-(I-T_1^*T_1).
\end{equation}
If this condition is satisfied then the following facts hold:
\begin{enumerate}\def\labelenumi{\rm(\roman{enumi})}
\item The completion problems for $A_{\pm}^0$ in
\eqref{E:A} have minimal solutions $A_\pm$.

\item The operators $T_m:=A_+-I$ and $T_M:=I-A_-\in
\Ext_{T_1,\kappa}(-1,1)$.

\item The operators $T_m$ and $T_M$ have the block form
\begin{equation}\label{E:T}
T_m=
\begin{pmatrix}
 T_{11}&D_{T_{11}}V^*\\
 VD_{T_{11}}&-I+V(I-T_{11})JV^*
\end{pmatrix},\
T_M=
\begin{pmatrix}
 T_{11}&D_{T_{11}}V^*\\
 VD_{T_{11}}&I-V(I+T_{11})JV^*
\end{pmatrix},
\end{equation}
where $D_{T_{11}}:=|I-T_{11}^2|^{1/2}$ and $V$ is given by
$V:=\clos(T_{21}D_{T_{11}}^{[-1]})$.

\item The operators $T_m$ and $T_M$ are extremal extensions of $T_1$:
\begin{equation}\label{E:Ext}
T\in\Ext_{T_1,\kappa}(-1,1)\  \text{  iff  }\  T=T^*\in[\sH],\quad
T_m\leq T\leq T_M.
\end{equation}

\item The operators $T_m$ and $T_M$ are connected via
\begin{equation}\label{E:5}
(-T)_m=-T_M, \quad (-T)_M=-T_m.
\end{equation}
\end{enumerate}
\end{theorem}
\begin{proof}
It is easy to see that $\kappa=\nu_-(I-T_{11}^2)\le
\nu_-(I-T_1^*T_1)\le \nu_-(I-T^2)$. Hence the condition
$\nu_-(I-T^2)=\kappa$ implies \eqref{crit}. The sufficiency of this
condition is established while proving the assertions (i)--(iii)
below.

(i) If the condition \eqref{crit} is satisfied then $\ran
T_{21}^*\subset \ran|I- T_{11}^2|^{1/2}$ by Lemma~\ref{corL1}. In
fact, this inclusion is equivalent to the inclusions $\ran
T_{21}^*\subset \ran|I\pm T_{11}|^{1/2}$, which by Theorem~\ref{T:1}
means that both of the completion problems, $A_{\pm}^0$ in
\eqref{E:A}, are solvable. Consequently, the following operators
\begin{equation}\label{E:S pm}
S_-=|I+T_{11}|^{[-1/2]}T_{21}^*,\quad
S_+=|I-T_{11}|^{[-1/2]}T_{21}^*
\end{equation}
are well defined and they provide the minimal solutions $A_\pm$ to
the completion problems for $A_\pm^0$ in \eqref{E:A}. Notice that the
assumption that there is a simultaneous solution $I\pm T$
with a \textit{single} selfadjoint operator $T$ is not yet used here.

(ii) \& (iii) Proof of (i) shows that the inclusion $\ran
T_{21}^*\subset \ran|I- T_{11}^2|^{1/2}$ holds. This last inclusion
alone is equivalent to the existence of a (unique) bounded operator
$V^*=D_{T_{11}}^{[-1]}T_{21}^*$ with $\ker V\supset \ker
D_{T_{11}}$, such that $T_{21}^*=D_{T_{11}}V^*$. The operators
$T_m:=A_+-I$ and $T_M:=I-A_-$ (see proof of (i)) can be now
rewritten as in \eqref{E:T}. Observe that
$$
S_\mp=|I\pm T_{11}|^{[-1/2]}D_{T_{11}}V^*=P_\mp|I\mp
T_{11}|^{1/2}V^*=|I\mp T_{11}|^{1/2}P_\mp V^*,
$$
where $P_\mp$ are the orthogonal projections onto
$$
(\ker|I\pm T_{11}|^{1/2})^\p=(\ker|I\pm T_{11}|)^\p=\ov{\ran}|I\pm
T_{11}|=\ov{\ran}|I\pm T_{11}|^{1/2}.
$$
Since $\ker V\supset \ker D_{T_{11}}$ implies $\ov{\ran}V^*\subset
\ov \ran D_{T_{11}}\subset \ov{\ran}|I\pm T_{11}|^{1/2}$, it follows
that
$$
S_-=|I-T_{11}|^{1/2}V^*,\quad S_+=|I+T_{11}|^{1/2}V^*.
$$
Consequently, see \eqref{JJpm},
$$
S_-^*J_-S_-=V|I-T_{11}|^{1/2}J_-|I-T_{11}|^{1/2}V^*=V(I-T_{11})JV^*,
$$
$$
S_+^*J_+S_+=V|I+T_{11}|^{1/2}J_+|I+T_{11}|^{1/2}V^*=V(I+T_{11})JV^*,
$$
which implies the representations for $T_m$ and $T_M$ in
\eqref{E:T}. Clearly, $T_m$ and $T_M$ are selfadjoint extensions of
$T_1$, which satisfy the equalities
$$
\nu_-(I+T_{m})=\kappa_-,\quad \nu_-(I-T_{M})=\kappa_+.
$$
Moreover, it follows from \eqref{E:T} that
\begin{equation}\label{E:6}
T_M-T_m=
\begin{pmatrix}
 0&0\\
 0&2(I-VJV^*)
\end{pmatrix}.
\end{equation}

Now the assumption \eqref{crit} will be used again. Since
$\nu_-(I-T_{1}^*T_{1})=\nu_-(I-T_{11}^2)$ and $T_{21}=VD_{T_{11}}$
it follows from Lemma~\ref{corL1} that $V^*\in[\sH_2,\sD_{T_{11}}]$
is $J$-contractive: $I-VJV^*\ge 0$. Therefore, \eqref{E:6} shows
that $T_M\geq T_m$ and $I+T_M\geq I+T_m$ and hence, in addition to
$I+T_m$, also $I+T_M$ is a solution to the problem $A_{+}^0$ and, in
particular, $\nu_-(I+T_M)=\kappa_-=\nu_-(I+T_m)$. Similarly,
$I-T_M\le I-T_m$ which implies that $I-T_m$ is also a solution to
the problem $A_{-}^0$, in particular,
$\nu_-(I-T_m)=\kappa_+=\nu_-(I-T_M)$. Now by applying Lemma
\ref{sign} we get
$$
\nu_-(I-T_m^2)=\nu_-(I-T_m)+\nu_-(I+T_m)=\kappa_++\kappa_-=\kappa,
$$
$$
\nu_-(I-T_M^2)=\nu_-(I-T_M)+\nu_-(I+T_M)=\kappa_++\kappa_-=\kappa.
$$
Therefore, $T_m,T_M\in \Ext_{T_1,\kappa}(-1,1)$ which in particular
proves that the condition \eqref{crit} is sufficient for solvability
of the completion problem \eqref{E:A}.

(iv) Observe, that $T\in\Ext_{T_1,\kappa}(-1,1)$ if and only if $T=T^*\supset T_1$
and $\nu_-(I\pm T)=\kappa_\mp$. By Theorem \ref{T:1} this is equivalent
to
\begin{equation}\label{E:7}
S_-^*J_-S_--I\leq T_{22}\leq I-S_+^*J_+S_+.
\end{equation}
The inequalities \eqref{E:7} are equivalent to \eqref{E:Ext}.

(v) The relations \eqref{E:5} follow from \eqref{E:S pm} and
\eqref{E:T}.
\end{proof}

For a Hilbert space contraction $T_1$ one has $\nu_-(I-T_{11}^2)\le
\nu_-(I-T_1^*T_1)=0$, i.e., the criterion \eqref{crit} is
automatically satisfied. In this case Theorem \ref{T:contr} has been
proved in \cite{HMS04}. As Theorem \ref{T:contr} shows, under the
minimal index condition $\nu_-(I-T^2)=\nu_-(I-T_{11}^2)$, the
solution set $\Ext_{T_1,\kappa}(-1,1)$ admits the same attractive
description as an operator interval determined by the two extreme
extensions $T_m$ and $T_M$ as was originally proved by M.G.
Kre\u{\i}n in his famous paper \cite{Kr47} when describing all
contractive selfadjoint extensions of a Hilbert space contraction.

In particular, Theorem \ref{T:contr} shows that if there is a solution to the completion problem \eqref{E:A},
i.e. if $T_1$ satisfies the index condition \eqref{crit}, then all selfadjoint extensions $T$
of $T_1$ satisfying the equality $\nu_-(I-T^2)= \nu_-(I-T_1^*T_1)$ are determined
by the operator inequalities $T_m\le T\le T_M$.

Notice that $T$ belongs to the solution set $\Ext_{T_1,\kappa}(-1,1)$
precisely when $T=T^*\supset T_1$ and $\nu_-(I\pm T)=\kappa_\mp$.
This means that every selfadjoint extension of $T_1$ for which $(I-T^2)=\nu_-(I-T_1^*T_1)$
admits precisely $\kappa_-$ eigenvalues on the interval $(-\infty,-1)$ and
$\kappa_+$ eigenvalues on the interval $(1,\infty)$; in total there are $\kappa=\kappa_- +\kappa_+$ eigenvalues outside the closed interval $[-1,1]$.
The fact that the numbers $\kappa_\mp=\nu_-(I\pm T)$ are constant in the solution
set $\Ext_{T_1,\kappa}(-1,1)$ is crucial for dealing properly with the Cayley transforms
in the next section.

\section{A generalization of M.G. Kre\u{\i}n's approach
to the extension theory of nonnegative operators}
\label{sec5}

\subsection{Some antitonicity theorems for selfadjoint
relations} 

The notion of inertia of a selfadjoint relation in a Hilbert space
is defined by means of its associated spectral measure. In what
follows the Hilbert space is assumed to be separable.

\begin{definition}
Let $H$ be a selfadjoint relation in a separable Hilbert space $\sH$
and let $E_t(\cdot)$ be the spectral measure of $H$. The inertia of
$H$ is defined as the ordered quadruplet
$\sfi(H)=\bigl\{\sfi^+(H),\sfi^-(H),\sfi^0(H),\sfi^\infty(H)\bigr\}$,
where
\begin{equation*}
\begin{split}
\sfi^+(H)&=\dim \ran E_t((0,\infty)),\qquad \sfi^-(H)=\dim \ran
E_t((-\infty,0)),\\
\sfi^0(H)&=\dim\ker H,\qquad\qquad\quad\, \sfi^\infty(H)=\dim \mul H.
\end{split}
\end{equation*}
\end{definition}
In particular, for a selfadjoint relation $H$ in $\dC^n$, the
quadruplet $\sfi(H)$ consists of the numbers of positive, negative,
zero, and infinite eigenvalues of $H$; cf. \cite{BHSW2014}. Hence,
if $H$ is a selfadjoint matrix in $\dC^n$, then ${\sf
i}^\infty(H)=0$ and the remaining numbers make up the usual inertia
of $H$.

The following theorem characterizes the validity of the implication
\[
 H_1\le H_2 \quad\Rightarrow \quad H_2^{-1}\le H_1^{-1}
\]
for a pair of bounded selfadjoint operators $H_1$ and $H_2$ having
bounded inverses; it in the infinite dimensional case has been
proved independently in \cite{S91,DM91b,HaNo0}; cf. also
\cite{HaNo}. Some extensions of this result, where the condition
$\min\{\sfi^+_2,\sfi^-_1\}<\infty$ is relaxed, are also contained in
\cite{S91,HaNo0,HaNo}.

\begin{theorem}\label{antith2}
Let $H_1$ and $H_2$ be bounded and boundedly invertible selfadjoint
operators in a separable Hilbert space $\sH$. Let
$\sfi(H_j)=\{\sfi^+_j,\sfi^-_j,\sfi^0_j,\sfi^\infty_j\}$ be the
inertia of $H_j$, $j=1,2$, and assume that
$\min\{\sfi^+_2,\sfi^-_1\}<\infty$ and that $H_1\leq H_2$. Then
\[
H_2^{-1} \leq H_1^{-1} \quad \textrm{if and only if}
\quad {\sf i}(H_1) = {\sf i}(H_2).
\]
\end{theorem}

Very recently two extensions of Theorem~\ref{antith2} have been
established in \cite{BHSW2014} for a general pair of selfadjoint
operators and relations without any invertibility assumptions. For
the present purposes we need the second main antitonicity theorem
from \cite{BHSW2014}, which reads as follows.

\begin{theorem}\label{antinew2}
Let $H_1$ and $H_2$ be selfadjoint relations in a separable Hilbert
space $\sH$ which are semibounded from below. Let
$\sfi(H_j)=\{\sfi^+_j,\sfi^-_j,\sfi^0_j,\sfi^\infty_j\}$ be the
inertia of $H_j$, $j=1,2$, and assume that $\sfi^-_1<\infty$ and
that $H_1 \leq H_2$. Then
\[
 H_2^{-1} \leq H_1^{-1} \quad \textrm{if and only if} \quad
 {\sf i}_1^- = {\sf i}_2^-.
 \]
\end{theorem}

The ordering appearing in Theorem~\ref{antinew2} is defined via
\[
 H_1 \leq H_2 \quad \Leftrightarrow \quad
 0\le (H_2-aI)^{-1}\le (H_1-aI)^{-1},
\]
where $a<\min\{\mu(H_1),\mu(H_2)\}$ is fixed and $\mu(H_i)\in\dR$
stands for the lower bound of $H_i$, $i=1,2$. Notice that the
conditions $H_1 \le H_2$ and $\sfi^-_1<\infty$ imply
$\sfi^-_2<\infty$; in particular these conditions already imply that
the inverses $H_1^{-1}$ and $H_2^{-1}$ are also semibounded from
below. For further facts on ordering of semibounded selfadjoint
operators and relations the reader is referred to
\cite{Kato,BHSW2014}.

\subsection{Cayley transforms}

Define the linear fractional transformation $\cC$, taking a linear
relation $A$ into a linear relation $\cC(A)$, by
\begin{equation}
\label{bilin}
  \cC(A)=\{\, \{f+f',f-f'\} :\, \wh f=\{f,f'\} \in A\,\}= -I+2(I+A)^{-1}.
\end{equation}
Clearly, $\cC$ maps the (closed) linear relations one-to-one onto
themselves, $\cC^{2}=I$, and
\begin{equation}
\label{bilin0}
 \cC(A)^{-1}=\cC(-A),
\end{equation}
for every linear relation $A$. Moreover,
\[
  \begin{split}
   &\dom \cC(A)=\ran (I+A), \quad \ran \cC(A)= \ran (I-A), \\
   &\ker (\cC(A)-I)=\ker A, \quad  \ker (\cC(A)+I)=\mul A.
  \end{split}
\]
In addition, $\cC$ preserves closures, adjoints, componentwise sums,
orthogonal sums, intersections, and inclusions. The relation
$\cC(A)$ is symmetric if and only if $A$ is symmetric. It follows
from \eqref{bilin} and
\begin{equation}
\label{bilin2}
 \|f+f'\|^2-\|f-f'\|^2=4\RE(f',f)
\end{equation}
that $\cC$ gives a one-to-one correspondence between nonnegative
(selfadjoint) linear relations and symmetric (respectively,
selfadjoint) contractions. Observe the following mapping properties
of $\cC$ on the extended real line $\dR\cup\{\pm\infty\}$:
\begin{equation}\label{Hin0}
 \begin{split}
  &\cC([0,1])=[0,1];\quad \cC([-1,0])=[1,+\infty];\\
  &\cC([1,+\infty])=[-1,0];\quad \cC([-\infty,-1])=[-\infty,-1].
 \end{split}
\end{equation}
If $H$ is a selfadjoint relation then
\[
 \sfi^-(I+H)=\sfi^-(\cC(H)+I),\qquad
 \sfi^-(I-H)=\sfi^-(\cC(H)^{-1}+I),
\]
and hence
\begin{equation}\label{Hin1}
 \begin{split}
 &\sigma(H)\cap (-\infty,-1)=\sigma(\cC(H))\cap (-\infty,-1),\\
 &\sigma(H)\cap (1,+\infty)=\sigma(\cC(H)^{-1})\cap (-\infty,-1)=\sigma(\cC(H))\cap (-1,0);
 \end{split}
\end{equation}
which can also be seen from \eqref{Hin0}.

\subsection{M.G. Kre\u{\i}n's approach
to the extension theory with a minimal negative index}

The crucial step in the M.G. Kre\u{\i}n's approach to the extension
theory of nonnegative operators is the connection to the selfadjoint
contractive extensions of a Hermitian contraction $T$ via the Cayley
transform in \eqref{bilin}. The extension of this approach to the
present indefinite situation is based on the fact that the Cayley
transform still reverses the ordering of selfadjoint extensions due
to the antitonicity result formulated in Theorem~\ref{antinew2} and
the fact that in Theorem~\ref{T:contr} $T\in\Ext_{T_1,\k}(-1,1)$ if
and only if $T=T^*\supset T_1$ and $\nu_-(I\pm T)=\k_\mp$.

A semibounded symmetric relation $A$ is said to be quasi-nonnegative
if the associated form $a(f,f):=(f',f)$, $\{f,f'\}\in A$, has a
finite number of negative squares, i.e. every $A$-negative subspace
$\sL\subset \dom A$ is finite-dimensional. If the maximal dimension
of $A$-negative subspaces is finite and equal to $\kappa\in\dZ_+$,
then $A$ is said to be $\kappa$-nonnegative; the more precise
notations $\nu_-(a)$, $\nu_-(A)$ are used to indicate the maximal
number of negative squares of the form $a$ and the relation $A$,
respectively; here $\nu_-(a)=\nu_-(A)$. A selfadjoint extension $\wt A$
of $A$ is said to be a $\kappa$-nonnegative extension of $A$ if
$\nu_-(\wt A)=\kappa$. The set of all such extension will be
denoted by $\Ext_{A,\kappa}(0,\infty)$.

If $A$ is a closed symmetric relation in the Hilbert space $\sH$
with $\kappa_-(A)<\infty$, then the subspace $\sH_1:=\ran(I+A)$ is
closed, since the Cayley-transform $T_1=\cC(A)$ is a closed bounded
symmetric operator in $\sH$ with $\dom T_1=\sH_1$. Let $P_1$ be the
orthogonal projection onto $\sH_1$ and let $P_2=I-P_1$. Then the
form
\begin{equation}
\label{defa1}
 a_1(f,f):=(P_1f',f), \quad \{f,f'\}\in A,
\end{equation}
is symmetric and it has a finite number of negative squares.

\begin{lemma}\label{a1lemma}
Let $A$ be a closed symmetric relation in $\sH$ with
$\kappa_-(A)<\infty$ and let $T_1=\cC(A)$. Then the form $a_1$ is
given by
\begin{equation}
\label{a1form}
  a_1(f,f)=a(f,f)+\|P_2 f \|^2
\end{equation}
with $\nu_-(a_1)\le \nu_-(A)$. Moreover,
\[
 4a_1(f,f)=\|g\|^2-\|T_{11}g\|^2,
 \quad
 4a(f,f)=\|g\|^2-\|T_1g\|^2,
\]
where $\{f,f'\}\in A$, $g=f+f'$, and $T_{11}=P_1T_1$. In addition,
$T_{21}=P_2T_1$ satisfies $\|T_{21}g \|^2=\|P_2f\|=-(P_2 f,f')$.
\end{lemma}
\begin{proof}
The formula \eqref{bilin2} shows that if $T_1=\cC(A)$ and $\{f,f'\}\in A$, then
\[
 \|g\|^2-\|T_1g\|^2=4 (f',f)=4 a(f,f), \quad g=f+f'\in \dom T_1= \sH_1.
\]
Moreover, $T_{21}g=P_2(f-f')=2P_2f=-2P_2f'$ gives $(P_2f',f)=-\|P_2 f\|^2$
and
\[
 \|T_{21}g \|^2=-4(P_2 f',P_2 f)=-4(P_2f',f).
\]
In particular, \eqref{a1form} follows from
\[
 a(f,f)=(P_1f',f)+(P_2f',f)=a_1(f,f)-\|P_2 f\|^2.
\]
Finally, \eqref{a1form} combined with $\|T_{21}g\|^2=4\|P_2f\|^2$
leads to
\[
 4a_1(f,f)=\|g\|^2-\|T_1g\|^2+\|T_{21}g\|^2=\|g\|^2-\|T_{11}g\|^2. \qedhere
\]
\end{proof}

The main result in this section concerns the existence and a description of all
selfadjoint extensions $\wt A$ of a symmetric relation $A$ for which
$\nu_-(\wt A)<\infty$ attains the minimal value $\nu_-(a_1)$.
A criterion for the existence of such a selfadjoint extension is established, in which case
all such extensions are described in a manner that is familiar
from the case of nonnegative operators. To formulate the result assume that
the selfadjoint quasi-contractive extensions $T_m$ and $T_M$ of
$T_1$ as in Theorem~\ref{T:contr} exist, and denote the corresponding selfadjoint
relations $A_F$ and $A_K$ by
\begin{equation}
\label{1.28}
  A_F = X(T_{m})=-I+2(I+T_m)^{-1},
\quad
  A_K = X(T_{M})=-I+2(I+T_M)^{-1}.
\end{equation}

\begin{theorem}
\label{KreinThm} Let $A$ be a closed symmetric relation in
$\sH$ with $\nu_-(A)<\infty$ and denote $\kappa=\nu_-(a_1)\,(\le \nu_-(A))$,
where $a_1$ is given by \eqref{defa1}.
Then $\Ext_{A,\kappa}(0,\infty)$ is nonempty if and only if $\nu_-(A)=\kappa$.
In this case $A_F$ and $A_K$ are well-defined and they belong to $\Ext_{A,\kappa}(0,\infty)$.
Moreover, the formula
\begin{equation}\label{CTA}
 \tA=-I+2(I+T)^{-1}
\end{equation}
gives a bijective correspondence between the quasi-contractive
selfadjoint extensions $T\in\Ext_{T_1,\k}(-1,1)$ of $T_1$ and the
selfadjoint extensions $\tA=\tA^*\in
\Ext_{A,\kappa}(0,\infty)$ of $A$. Furthermore, $\tA=\tA^*\in
\Ext_{A,\kappa}(0,\infty)$ precisely when
\begin{equation}\label{AKAF}
 A_K\le \wt A\le A_F,
\end{equation}
or equivalently, when $A_F^{-1}\le \wt A^{-1}\le A_K^{-1}$, or
\begin{equation}\label{resolA}
 (A_F+I)^{-1}\le (\tA+I)^{-1}\le (A_K+I)^{-1}.
\end{equation}
The set $\Ext_{A^{-1},\kappa}(0,\infty)$ is also nonempty and $\wt A\in \Ext_{A,\kappa}(0,\infty)$
if and only if $\wt A^{-1}\in \Ext_{A^{-1},\kappa}(0,\infty)$.
The extreme selfadjoint extensions $A_F$ and $A_K$ of $A$ are connected to those of $A^{-1}$ via
\begin{equation}
\label{symme}
 (A^{-1})_F=(A_K)^{-1}, \quad (A^{-1})_K=(A_F)^{-1}.
\end{equation}
\end{theorem}

\begin{proof}
Since $\nu_-(A)<\infty$ the Cayley transform $T_1=\cC(A)$ defines a bounded symmetric
operator in $\sH$ with $\sH_1=\dom T_1=\ran (I+A)$.
It follows from Lemma \ref{a1lemma} that
\[
 \nu_-(A)=\nu_-(a)=\nu_-(I-T_1^*T_1), \quad \nu_-(a_1)=\nu_-(I-T_{11}^2),
\]
and therefore the condition $\nu_-(A)=\kappa$ is equivalent to solvability criterion
\eqref{crit} in Theorem~\ref{T:contr}. Moreover, $\wt A$ is a selfadjoint extension
of $A$ if and only if $T=\cC(\wt A)$ is selfadjoint extension of $T_1$ and by Lemma~\ref{a1lemma}
the equality $\nu_-(\wt A)=\nu_-(I-T^2)$ holds. Therefore, it follows from Theorem~\ref{T:contr}
that the set $\Ext_{A,\kappa}(0,\infty)$ is nonempty if and only if $\nu_-(A)=\kappa$ and in
this case the formula \eqref{CTA} establishes a one-to-one correspondence between the
sets $\Ext_{A,\kappa}(0,\infty)$ and $\Ext_{T_1,\k}(-1,1)$.

Next the characterizations \eqref{AKAF} and \eqref{resolA} for the
set $\Ext_{A,\kappa}(0,\infty)$ are established. Let $\wt A\in
\Ext_{A,\kappa}(0,\infty)$ and let $T=\cC(\wt A)$. According to
Theorem~\ref{antinew2} $T=\cC(\wt A)\in \Ext_{T_1,\k}(-1,1)$ if and
only if $T$ satisfies the inequalities $T_m\le T\le T_M$. It is
clear from the formulas \eqref{1.28} and \eqref{CTA} that the
inequalities $I+T_m\le I+T\le I+T_M$ are equivalent to the
inequalities \eqref{resolA}.

On the other hand, $\nu_-(I-T_{11}^2)=\nu_-(I-T^2)$ and hence the indices $\kappa_+=\nu_-(I-T_{11})=\nu_-(I-T)$ and
$\kappa_-=\nu_-(I+T_{11})=\nu_-(I+T)$ do not depend on $T=\cC(\wt A)$; cf. \eqref{JJpm}.
The mapping properties \eqref{Hin1} of the Cayley transform imply that the number
of eigenvalues of $\wt A$ on the open intervals $(-\infty,-1)$ and $(-1,0)$ are also constant and
equal to $\kappa_-$ and $\kappa_+$, respectively.
In particular, since $\kappa_-=\nu_-(I+T)$ is constant we can apply
Theorem~\ref{antith2} to conclude that the inequalities $I+T_m\le I+T\le I+T_M$ are equivalent to
\[
 (I+T_M)^{-1}\le (I+T)^{-1}\le (I+T_m)^{-1},
\]
which due to the formulas \eqref{1.28} and \eqref{CTA} can be rewritten as
$A_F+I\le \wt A+I \le A_K+I$, or as $A_F\le \wt A\le A_K$. This proves \eqref{AKAF}.
Since $\nu_-(\wt A)=\kappa=\kappa_-+\kappa_+$ is also constant,
an application of Theorem~\ref{antinew2} shows that the inequalities \eqref{AKAF}
are also equivalent to $A_F^{-1}\le \wt A^{-1}\le A_K^{-1}$.

As to the inverse $A^{-1}$, notice that $\nu_-(A^{-1})=\nu_-(A)$.
Moreover, since $A^{-1}=\cC(-T_1)$ it is clear that $\ran
(I+A^{-1})=\dom T_1$ and thus the form associated to $A^{-1}$ via
\eqref{defa1} satisfies
$a_1^{(-1)}(f',f')=(P_1f,f')=(P_1f',f)=a_1(f,f)$. In particular,
$\nu_-(a_1^{(-1)})=\nu_-(a_1)$. Moreover, it is clear that
$\nu_-(A^{-1})=\nu_-(A)$. Consequently, the equality
$\nu_-(A)=\nu_-(a_1)$ is equivalent to the equality
$\nu_-(A^{(-1)})=\nu_-(a_1^{(-1)})$. Furthermore, it is clear that
$\wt A\in \Ext_{A,\kappa}(0,\infty)$ if and only if $\wt A^{-1}\in
\Ext_{A^{-1},\kappa}(0,\infty)$.

Finally, the relations \eqref{symme} are obtained from
\eqref{E:5}, \eqref{bilin0}, and \eqref{1.28}.
\end{proof}

It follows from Theorem~\ref{KreinThm} that the extensions $\tA \in
\Ext_{A,\kappa}(0,\infty)$ admit a uniform lower bound $\mu\le
\mu(\wt A)\,(\mu\le 0)$. Consequently, the resolvents of these
extensions satisfy
\begin{equation}
\label{inequ}
 (A_F+a)^{-1}\le (\tA + a)^{-1}\le (A_K+a)^{-1},
 \quad
 a>-\mu.
\end{equation}
This follows from the formula
\begin{equation*}
    (\tA+a)^{-1}
    = \frac{1}{a-1}\, I-\frac{2}{(a-1)^2}\,
      \left(T+\frac{a+1}{a-1}\right)^{-1}
\end{equation*}
and the fact that $T=\cC(\wt A)\in \Ext_{T_1,\k}(-1,1)$ has
precisely $\kappa_-$ eigenvalues below the number
$-{(a+1)}/{(a-1)}<-1$, so that the inequalities $T_m\leq T\leq T_M$
in Theorem~\ref{T:contr} imply the inequalities \eqref{inequ} by
Theorem~\ref{antith2}.

The antitonicity Theorems~\ref{antith2},~\ref{antinew2} can be also
used as follows. If the inequalities \eqref{AKAF} and $A_F^{-1}\le
\wt A^{-1}\le A_K^{-1}$ hold, then $\kappa=\nu_-(\wt
A)=\nu_-(A_K)=\nu_-(A_F)$ is constant. If, in addition,
\eqref{resolA} is satisfied, then it follows from \eqref{AKAF} that
$\kappa_-=\nu_-(I+\wt A)=\nu_-(I+A_K)=\nu_-(I+A_F)$ is constant, so
that also $\kappa_+=\nu_-(I-\wt A)=\nu_-(I-A_K)=\nu_-(I-A_F)$ is
constant. However, in this case the equality $\nu_-(a_1)=\nu_-(A)$
need not hold and there can also be selfadjoint extensions $\wt A$
of $A$ with
\[
 \nu_-(\wt A)=\nu_-(A_K)=\nu_-(A_F)>\nu_-(A)\ge \nu_-(a_1),
\]
which neither satisfy the inequalities \eqref{AKAF} and
\eqref{resolA}, nor the equalities $\nu_-(I+\wt A)=\kappa_-$ and
$\nu_-(I-\wt A)=\kappa_+$. It is emphasized that the result in
Theorem~\ref{KreinThm} characterizes all selfadjoint extensions in
$\Ext_{A,\kappa}(0,\infty)$ under the minimal index condition
$\kappa=\nu_-(a_1)=\nu_-(A)$.

In the case that $A$ is nonnegative one has automatically
$\kappa=\nu_-(a_1)=\nu_-(A)=0$. Therefore, Theorem~\ref{KreinThm} is
a precise generalization of the famous characterization of the class
$\Ext_{A}(0,\infty)$ (with $\kappa=0$) due to M.G. Kre\u{\i}n
\cite{Kr47} to the case of a finite negative (minimal) index
$\kappa>0$. The selfadjoint extensions $A_F$ and $A_K$ of $A$ are
called the Friedrichs (hard) and the Kre\u{\i}n-von Neumann (soft)
extension, respectively; these notions go back to \cite{F,JvN}. The
extremal properties \eqref{inequ} of the Friedrichs and
Kre\u{\i}n-von Neumann extensions were discovered by Kre\u{\i}n
\cite{Kr47} in the case when $A$ is a densely defined nonnegative
operator. The case when $A\ge 0$ is not densely defined was
considered by T.~Ando and K.~Nishio \cite{AN}, and E.A.~Coddington
and H.S.V.~de~Snoo \cite{CS}.
In the nonnegative case the formulas \eqref{symme} can be found in
\cite{AN} and \cite{CS}.

\subsection{Kre\u{\i}n's uniqueness criterion}

To establish a generalization of Kre\u{\i}n's uniqueness criterion for the equality
$A_F=A_K$ in Theorem~\ref{KreinThm}, i.e., for $\Ext_{A,\kappa}(0,\infty)$ to
consists only of one extension, we first derive some general facts on $J$-contractions
by means of their commutation properties.

Let $\sH_1$ and $\sH_2$ be Hilbert spaces with symmetries $J_1$ and $J_2$, respectively,
and let $T\in[\sH_1,\sH_2]$ be a $J$-contraction, i.e., $J_1-T^*J_2T\geq 0$.
Let $D_T$ and $D_{T^*}$ be the corresponding defect operators and
let $J_T$ and $J_{T^*}$ be their signature operators as defined in Section \ref{sec3}.
The first lemma connects the kernels of the defect operators $D_T$ and $D_{T^*}$.

\begin{lemma}
\label{lm0} Let $T\in [\sH_1,\sH_2]$, let $J_i$ be a symmetry in $\sH_i$, $i=1,2$,
and let $D_T$ and $D_{T^*}$ be the defect operators of $T$ and $T^*$, respectively.
Then
\begin{equation}
\label{lm01}
  J_2T(\ker D_{T})=\ker D_{T^*},
  \quad
  T^*J_2(\ker D_{T^*})=\ker D_{T}.
\end{equation}
In particular,
\[
 \ker D_{T}=\{0\} \mbox{ if and only if } \ker D_{T^*}=\{0\}.
\]
\end{lemma}

\begin{proof}
It suffices to show the first identity in~\eqref{lm01}. If $\ff \in
\ker D_T=\ker J_TD_T^2$, then the second identity in \eqref{eqC3}
implies that $J_2T \ff \in \ker J_{T^*}D_{T^*}^2=\ker D_{T^*}$.
Hence, $J_2T(\ker D_{T}) \subset \ker D_{T^*}$. Conversely, let $\ff
\in \ker D_{T^*}$. Then $0=J_{T^*}D_{T^*}^2\ff$ or, equivalently,
$\ff = J_2TJ_1T^*\ff$, and here $J_1T^*\ff \in \ker D_T$ by the first
identity in \eqref{eqC3}. This proves the reverse inclusion.
\end{proof}

\begin{lemma}
\label{lm2} Let the notations be as in Lemma~\ref{lm0}. Then
\[
 \ran T \cap \ran D_{T^*}=\ran TJ_1D_T=\ran D_{T^*}L_T,
\]
where $L_T$ is the link operator defined in Corollary~\ref{linkoper}.
\end{lemma}
\begin{proof}
By the commutation formulas in Corollary~\ref{linkoper} $\ran TJ_1D_T=\ran
D_{T^*}L_T \subset \ran T \cap \ran D_{T^*}$. Hence, it
suffices to prove the inclusion
\[
 \ran T \cap \ran D_{T^*} \subset \ran TJ_1D_T.
\]
Suppose that $\varphi \in  \ran T \cap \ran D_{T^*}$.
Then Corollary~\ref{linkoper} shows that $T^*J_2\varphi = D_T f$ for some $f\in \sD_T$,
while the second identity in \eqref{eqC3} implies that
\[
 (J_2-TJ_1T^*)J_2\ff=TJ_1D_Tg,
\]
for some $g\in \sD_T$. Therefore,
\[
 \varphi=(J_2-TJ_1T^*)J_2\varphi+TJ_1T^*J_2\varphi
 =TJ_1D_Tg+TJ_1D_Tf=TJ_1D_T(g+f)
\]
and this completes the proof.
\end{proof}

We can now characterize $J$-isometric operators $T\in [\sH_1,\sH_2]$ as follows.

\begin{proposition}
\label{lm3} With the notations as in Lemma~\ref{lm0} the
following statements are equivalent:
\begin{enumerate}
\def\labelenumi{\rm (\roman{enumi})}
\item $T$ is $J$-isometric, i.e., $T^*J_2T=J_1$;
\item $\ker T=\{0\}$ and $\ran T \cap \ran D_{T^*} =\{0\}$;
\item for some, and equivalently for every, subspace $\sL$ with
$\ran J_2T\subset \sL$ one has
\begin{equation}
\label{apu0}
  \sup_{f\in\sL}\frac{|(f,T\varphi)|}{\|D_{T^*}f\|}=\infty
\quad \mbox{for every } \varphi\in \sH_1\backslash\{0\},
\end{equation}
i.e., there is no constant $0\le C<\infty$ satisfying $|(f,T\varphi)|\le C \|D_{T^*}f\|$ for every $f\in \sL$,
if $\varphi\neq 0$.
\end{enumerate}
\end{proposition}
\begin{proof}
(i) $\Rightarrow$ (iii) Let $\sL$ be an arbitrary subspace with
$\ran J_2T\subset\sL$. Assume that the supremum in \eqref{apu0} is
finite for some $\varphi=J_1\psi\in \sH_1$.
Then there exists $0\le C<\infty$, such
that
\[
  |(f,TJ_1\psi)|\le C\|D_{T^*}f\| \quad \mbox{ for every } f\in \sL.
\]
Since $\ran J_2T\subset \sL$ and $T$ is $J$-isometric, also the following inequality holds:
\begin{equation}
\label{apu02}
  \|\psi\|^2=(J_1T^*J_2T\psi,\psi) \le C\|D_{T^*}J_2T\psi\|.
\end{equation}
By taking adjoints (and zero extension for $L_{T^*}$) in the second identity in Corollary~\ref{linkoper}
it is seen that $D_{T^*}J_2T\psi=L_{T^*}^*D_T \psi=0$, since $T$ is
$J$-isometric. Hence \eqref{apu02} implies $\varphi=J_1\psi=0$.
Therefore \eqref{apu0} holds for every $\varphi\neq 0$.

(iii) $\Rightarrow$ (ii) Assume that \eqref{apu0} is satisfied with
some subspace $\sL$. If (ii) does not hold, then either $\ker T\neq
\{0\}$, in which case \eqref{apu0} does not hold for $0\neq
\varphi\in \ker T$, or $\ran T\cap \ran D_{T^*}\neq \{0\}$. However,
then with $0\neq T\varphi=D_{T^*}h$ the supremum in \eqref{apu0} is
finite even if $f$ varies over the whole space $\sH_2$. Thus, if
(ii) does not hold then \eqref{apu0} fails to be true.

(ii) $\Rightarrow$ (i) Let $\ran T \cap \ran D_{T^*} =\{0\}$. Then
by Lemma \ref{lm2} $TJ_1D_T=0$ and it follows from $\ker{T}=\{0\}$ that
$D_T=0$, i.e., $T$ is isometric. This completes the proof.
\end{proof}

After these preparations we are ready to prove the analog of Kre\u{\i}n's uniqueness
criterion for the equality $T_{m}=T_{M}$ in the case of quasi-contractions appearing
in Theorem~\ref{T:contr}.

\begin{theorem}
\label{pr1.2} Let the Hilbert space $\sH$ be decomposed as $\sH=\sH_1\oplus \sH_2$
and let $T_1 \in [\sH_1,\sH]$ be a symmetric quasi-contraction satisfying the condition
\eqref{crit} in Theorem~\ref{T:contr}. Then $T_{m}=T_{M}$ if and only if
\begin{equation}
\label{1.20}
  \sup_{f\in \sH_1}\frac{|(T_1f,\ff)|^2}{(|I-T_1^*T_1|f,f)}=\infty
  \quad
  \mbox{for every } \ff \in\sH_2\setminus \{0\}.
\end{equation}
\end{theorem}

\begin{proof}
Let $J=\sign(I-T_{11}^2)$.
According to Theorem \ref{T:contr} there is $V\in[\sD_{T_{11}},\sH_2]$,
such that $T_{21}=VD_{T_{11}}$; moreover, $V^*$ a $J$-contraction, i.e.,
$I-VJV^*\ge 0$. This implies that
\begin{equation}
\label{apu03}
 (T_1f,\ff)=(T_{21}f,\ff)=(D_{T_{11}}f,V^*\ff),
\end{equation}
and a direct calculation shows that
\begin{equation}
\label{apu04}
 I-T^*_1T_1=I-T_{11}^2-T_{21}^*T_{21}=JD_{T_{11}}^2-D_{T_{11}}V^*VD_{T_{11}}=
 D_{T_{11}}D_VJ_V D_VD_{T_{11}}.
\end{equation}
By construction $D_V\in[\sD_{T_{11}}]$ and therefore $\ran D_VD_{T_{11}}$ is dense in $\sD_V=\cran D_V$.
Furthermore, since $V^*$ is $J$-contractive it follows from Lemma~\ref{inertia} that
$\nu_-(J_V)=\nu_-(J)=\nu_-(I-T_{11}^2)$ and, therefore, the assumption \eqref{crit}
shows that $\nu_-(J_V)=\nu_-(I-T_1^*T_1)$. Now according to Proposition \ref{BKcor1} (ii)
if follows from \eqref{apu04} that there is a unique $J$-unitary operator $C\in[\sD_{T_1},\sD_V]$
such that $D_VD_{T_{11}}=CD_{T_{1}}$.

In view of \eqref{E:6} $T_m=T_M$ if and only if $V^*$ is
$J$-isometric. Since $\ran JV^*\subset \ov\ran D_{T_{11}}$, it
follows from Proposition \ref{lm3} that $T:=V^*$ satisfies the
condition \eqref{apu0} with $\sL=\ov\ran D_{T_{11}}$.

On the other hand, it follows from \eqref{apu04} and the $J$-unitarity of $C\in[\sD_{T_1},\sD_V]$ that
\[
  \|D_VD_{T_{11}}\|\le \|C\|\, \|D_{T_1}\|,\quad \|D_{T_1}\|\le \|C^{-1}\|\, \|D_VD_{T_{11}}\|.
\]
By combining this equivalence between the norms of $\|D_{T_1}\|$ and $\|D_VD_{T_{11}}\|$ with
the equality \eqref{apu03} one concludes that $V^*$ satisfies the condition \eqref{apu0}
precisely when $T_1$ satisfies the condition \eqref{1.20}.
\end{proof}

This result can be translated to the situation of Theorem
\ref{KreinThm} via Cayley transform to get the analog of
Kre\u{\i}n's uniqueness criterion for the equality $A_F=A_K$.

\begin{corollary}
\label{Aunique} Let $A$ be a closed symmetric relation in $\sH$
satisfying the condition $\nu_-(A)=\nu_-(a_1)<\infty$ in Theorem
\ref{KreinThm}. Then the equality $A_{F}=A_{K}$ holds if and only if
the following condition is fulfilled:
\begin{equation}
\label{1.20A}
  \sup_{g\in \sH_1}\frac{|((A+I)^{-1}g,\ff)|^2}{(|\widehat{A}|g,g)}=\infty
  \quad
  \mbox{for every } \ff \in\ker (A^*+I)\setminus \{0\},
\end{equation}
where $\widehat{A}=(I+A)^{-*}A(I+A)^{-1}$ is a bounded selfadjoint operator in $\sH_1=\ran (A+I)$.
\end{corollary}

\begin{proof}
Let $T_1=\cC(A)$ so that $\{f,f'\}\in A$ if and only if $\{f+f',2f\}\in T_1+I$; see \eqref{bilin}.
Then with $g=f+f'\in\dom T_1=\sH_1$ and $\varphi\in \sH_2=(\dom T_1)^\perp$ one has
\[
 (T_1g,\varphi)=((T_1+I)g,\varphi)=2((A+I)^{-1}g,\varphi). 
\]
Let $A_s=P_sA$ be the operator part of $A$; here $P_s$ stand for the orthogonal projection onto
$\mul A=(\dom A^*)^\perp=\ker (T_1+I)$. Then the form $a(f,f)=(f',f)$ associated with $A$ can be rewritten
as $a(f,f)=(A_sf,f)$, $f\in \dom A$, and thus
\[
 ((I-T_1^*T_1)g,g)=4(f',f)=4(A_s(I+A)^{-1}g,(I+A)^{-1}g)),
\]
where $2(I+A)^{-1}=T_1+I$ is a bounded operator from $\sH_1$ into $\sH$.
Then clearly $\widehat{A}=(I+A)^{-*}A_s(I+A)^{-1}$ is a bounded selfadjoint operator in $\sH_1$
and, moreover, $\nu_-(\widehat{A})=\nu_-(a)=\nu_-(I-T_1^*T_1)$; see Lemma~\ref{a1lemma}.
Thus, it follows from Proposition \ref{BKcor1} that there is a $J$-unitary operator $C$ from
$\cran \widehat{A}$ into $\sD_{T_1}$ such that $D_{T_1}=C|\widehat{A}|^{1/2}$.
As in the proof of Theorem \ref{KreinThm} this implies the equivalence of the conditions
\eqref{1.20} and  \eqref{1.20A}.
\end{proof}

Observe that if $A$ is nonnegative then with $\{f,f'\}\in A$ and $g=f+f'\in\sH_1$,
\[
 ((A+I)^{-1}g,\varphi)=(f,\varphi), \quad (A_s(I+A)^{-1}g,(I+A)^{-1}g))=(A_sf,f),
\]
and, therefore, in this case the condition \eqref{1.20A} can be rewritten as
\[
   \sup_{\{f,f'\}\in A} \frac{|(f,\varphi)|^2}{(f',f)}=\infty
  \quad
  \mbox{for every } \varphi \in\ker (A^*+I)\setminus \{0\},
\]
the criterion which for a densely defined operator $A$ was obtained
in \cite{Kr47} and for a nonnegative relation $A$ can be found from
\cite{HMS04,Ha04}.

\bigskip

\noindent \textbf{Acknowledgements.} Main part of this work was
carried out during the second author's sabbatical year 2013-2014; he
is grateful about the Professor Pool's grant received from the Emil
Aaltonen Foundation.

\end{document}